\documentclass[10pt,a4paper]{amsart}
\usepackage{amsmath,amsfonts,amssymb,amscd,amsthm,amsbsy,bbm, epsf,calc,enumerate,color,graphicx,verbatim,cite}
\usepackage{color}
\usepackage{datetime}
\usepackage{hyperref}
\usepackage[usenames,dvipsnames]{xcolor}
\usepackage{endnotes}

\usepackage{url}

\usepackage[disable]{todonotes}

\let\footnote=\endnotetext

\usepackage{needspace,calc}
\newlength{\heightRecaller}
\newcommand{\reservespace}[1]{%
    \let \oldstepcounter \stepcounter%
    \renewcommand{\stepcounter}[1]{}%
    \settototalheight{\heightRecaller}{\parbox{\textwidth}{#1}}%
    \let \stepcounter \oldstepcounter%
    \needspace{\heightRecaller+4\baselineskip}%
    #1%
}

\usepackage{esint}

\usepackage{tikz}
\usepackage{caption}
\usepackage{subcaption}
\graphicspath{{./graphics/}}

\usepackage{comment}
\excludecomment{figure}

\def\Xint#1{\mathchoice
{\XXint\displaystyle\textstyle{#1}}%
{\XXint\textstyle\scriptstyle{#1}}%
{\XXint\scriptstyle\scriptscriptstyle{#1}}%
{\XXint\scriptscriptstyle\scriptscriptstyle{#1}}%
\!\int}
\def\XXint#1#2#3{{\setbox0=\hbox{$#1{#2#3}{\int}$ }
\vcenter{\hbox{$#2#3$ }}\kern-.6\wd0}}
\def\ddashint{\Xint=}
\def\dashint{\Xint-}

\newcounter{hours}\newcounter{minutes}

\theoremstyle{theorem}

\newtheorem{thm}{Theorem}[section]

\newtheorem{lem}[thm]{Lemma}

\newtheorem{cor}[thm]{Corollary}
\newtheorem{prop}[thm]{Proposition}

\newtheorem{THM}{Theorem}

\theoremstyle{definition}
\newtheorem{ex}[thm]{Example}

\newtheorem{DEF}[thm]{Definition}

\theoremstyle{remark}                  

\newtheorem{rem}[thm]{Remark}

\makeatletter
\theoremstyle{theorem}
\newtheorem*{rep@theorem}{\rep@title}
\newcommand{\newreptheorem}[2]{%
	\newenvironment{rep#1}[1]{%
		\def\rep@title{#2 \ref{##1}}%
		\begin{rep@theorem}}%
		{\end{rep@theorem}}}
\makeatother

\newreptheorem{thm}{Theorem}
\newreptheorem{lem}{Lemma}
\newreptheorem{cor}{Corollary}
\newreptheorem{prop}{Proposition}
\newreptheorem{DEF}{Definition}

\makeatletter
\theoremstyle{definition}
\newtheorem*{REP@theorem}{\rep@title}
\newcommand{\newREPtheorem}[2]{%
	\newenvironment{REP#1}[1]{%
		\def\rep@title{#2 \ref{##1}}%
		\begin{REP@theorem}}%
		{\end{REP@theorem}}}
\makeatother

\newREPtheorem{DEF}{Definition}

\def\C{{\mathcal C}}
\def\D{{\mathcal D}}

\def\E{{\mathcal E}}

\def\F{{\mathcal F}}
\def\G{{\mathcal G}}
\def\H{{\mathcal H}}
\def\I{{\mathcal I}}
\def\J{{\mathcal J}}
\def\K{{\mathcal K}}
\def\L{{\mathcal L}}
\def\M{{\mathcal M}}
\def\N{{\mathcal N}}

\def\R{{\mathbb R}}

\def\W{{\mathcal W}}

\def\m{{\overline{m}}}

\def\integer{{\mathbb Z}}
\def\natural{{\mathbb N}}

\def\e{\varepsilon}

\def\Om{\Omega}

\def\gam{\gamma} 
 
\def\lam{\lambda}
\def\Lam{\Lambda}
\def\z{\zeta}
\def\Up{\Upsilon}

\def\grad{\nabla}

\DeclareMathOperator*{\argmin}{\arg\!\min}
\def\liminf{\mathop{\lim\,\inf}\limits}%
\def\liminfs{\mathop{\lim\,\inf{}_*}\limits}%
\def\limsup{\mathop{\lim\,\sup}\limits}%
\def\limsups{\mathop{\lim\,\sup{}^*}\limits}%
\def\argmin{\mathop{\arg\,\min}\limits}%

\numberwithin{equation}{section}

\usepackage[utf8]{inputenc}
\usepackage{color}

\usepackage{amsmath}
\usepackage{amssymb}
\usepackage{amsthm}
\usepackage{fullpage}
\usepackage{enumerate}
\usepackage{dsfont}
\usepackage{stmaryrd}
\usepackage{graphicx}
\usepackage{mathrsfs}

\def\per{\textnormal{Per}}
\def\tr{\textnormal{trace}}

\newcommand{\oB}{{\overline B}}
\newcommand{\oD}{{\overline D}}

\newcommand{\oL}{{\overline L}}

\newcommand{\cD}{{\mathcal D}}

\newcommand{\cI}{{\mathcal I}}
\newcommand{\cJ}{{\mathcal J}}
\newcommand{\cK}{{\mathcal K}}

\newcommand{\td}{{\widetilde d}}

\newcommand{\tu}{{\widetilde u}}

\newcommand{\tv}{{\widetilde v}}

\newcommand{\hr}{{\widehat r}}
\newcommand{\hu}{{\widehat u}}

\newcommand{\hzm}{{\widehat {\z^-}}}

\newcommand{\Rn}{{\mathbb R^n}}
\newcommand{\Rpz}{{\mathbb R^+_0}}
\newcommand{\Rpt}{{[0,T]}}

\newcommand{\weakly}{\ensuremath{\rightharpoonup}}

\newcommand{\pot}{{\partial \Omega_t}}

\newcommand{\potj}{{\partial \Omega_t^j}}
\newcommand{\poz}{{\partial \Omega_0}}

\newcommand{\pet}{{\partial E_t}}
\newcommand{\pe}{{\partial E}}
\newcommand{\pdr}{{\partial_p D_r}}

\newcommand{\podt}{{\partial \Omega^\delta_t}}

\newcommand{\ot}{{\Omega_t}}

\newcommand{\ott}{{(\ot)_{t\geq0}}}
\newcommand{\oitt}{{(\oit)_{t\geq0}}}
\newcommand{\oktt}{{(\okt)_{t\geq0}}}
\newcommand{\odtt}{{(\odt)_{t\geq0}}}

\newcommand{\odt}{{\Omega^\delta_t}}

\newcommand{\okt}{{\Omega^{k}_{t}}}

\newcommand{\oz}{{\Omega_0}}
\newcommand{\oit}{{\Omega^\infty_t}}

\newcommand{\odta}{{\Omega^\delta_{t_1}}}
\newcommand{\odtb}{{\Omega^\delta_{t_2}}}
\newcommand{\ota}{{\Omega_t^1}}
\newcommand{\otb}{{\Omega_t^2}}
\newcommand{\otj}{{\Omega_t^j}}

\newcommand{\ottj}{{(\otj)_{t\geq0}}}

\newcommand{\uit}{{U^\infty_t}}
\newcommand{\uitt}{{(\uit)_{t\geq0}}}

\newcommand{\ei}{{\eta^\infty}}

\newcommand{\lt}{\lambda(t)}
\newcommand{\ld}{\lambda_{\delta}}

\newcommand{\ldh}{\lambda_{\delta}^h}
\newcommand{\ldt}{\lambda_{\delta}(t)}

\newcommand{\ldth}{\lambda_{\delta}^h(t)}
\newcommand{\rr}{\rho\text{-reflection}}

\newcommand{\li}{{\lambda_\infty}}

\newcommand{\Li}{{\Lambda_{\infty}}}
\newcommand{\Ld}{{\Lambda_{\delta}}}
\newcommand{\Lk}{{\Lambda_{k}}}
\newcommand{\Lki}{{\Lambda_{k_i}}}

\newcommand{\cJd}{\cJ_{\delta}}
\newcommand{\Jd}{\J_{\delta}}

\newcommand{\gd}{\gamma_{\delta}}

\newcommand{\fai}{for all $i \in \natural$}

\newcommand{\sd}{d_{signed}}

\newcommand{\n}{\vec{n}}

\newcommand{\dv}{\textnormal{div}}

\newcommand{\Gp}{{\G_s[\phi]}}
\newcommand{\Gps}{{\G_s[\psi]}}

\newcommand{\resp}{{(supersolution, respectively) }}

\newcommand{\Brt}{{B_r^{n-1}(0) \times [0,T]}}

\begin{document}

\title{Volume preserving mean curvature flow for star-shaped sets}
\author{Inwon Kim}
\address{Department of Mathematics, UCLA, Los Angeles, USA}
\thanks{Inwon Kim was partially supported by NSF DMS-1566578.}
\email{ikim@math.ucla.edu}

\author{Dohyun Kwon}
\address{Department of Mathematics, UCLA, Los Angeles, USA}
\email{dhkwon@ucla.edu}
\thanks{Dohyun Kwon was partially supported by Kwanjeong Educational Foundation.}
\date{}

\subjclass[2010]{}


\begin{abstract}
We study the evolution of star-shaped sets in volume preserving mean curvature flow. Constructed by approximate minimizing movements, our solutions preserve a strong version of star-shapedness. We also show that the solutions converges to a ball as time goes to infinity. For asymptotic behavior of the solutions we use the gradient flow structure of the problem, whereas a modified notion of viscosity solutions is introduced  to study the geometric properties of the flow by moving planes method. 

 \end{abstract}

\maketitle

\section{Introduction}

Let $\Omega_0$ be a open and bounded domain in $\R^n$ with unit volume, and consider the evolution of sets $(\Omega_t)_{t\geq0}$ moving with the normal velocity $V$ given by 
\begin{align}
\label{model}
V=-H+\lambda(t) \hbox{ on } \Gamma_t:= \partial\Omega_t, \quad |\Omega_t| = |\Omega_0|.
\end{align}
In smooth setting, $H$ is mean curvature on $\pot$ where $H$ is set to be positive if the domain is convex at the point, and $\lambda : \Rpz \rightarrow \R$ satisfies $\int_{\Gamma_t} V dS = 0$ so that  the evolution satisfies  $|\Omega_t| = |\Omega_0|$, i.e. 
\begin{align}
\label{lag}
\lambda(t) = \frac{1}{\per(\Omega_t)}\int_{\partial \Omega_t} H d\sigma = \fint_{\partial \Omega_t} H d\sigma.
\end{align}

There are two main difficulties to study the global behavior of the flow \eqref{model} in general settings. First the evolution may go through topological changes, and secondly the formula \eqref{lag} does not hold for $\lam$ in less than $C^{1,\alpha}$ settings. The first difficulty motivates us to study geometric properties that are preserved by the flow, and the second requires new ideas to obtain sufficient compactness to establish convergence to equilibrium.

\medskip

In variational setting,  \eqref{model} can be formulated based on its energy dissipation structure for the perimeter energy with volume preserving constraint. Using this structure \cite{MugSeiSpa16} and \cite{Tak17} showed the existence of general distribution solution of \eqref{model}. For our interest in geometric properties of solutions, we instead work with a modified version of viscosity solutions, where we consider an implicit choice of $\lam$ so that the volume of the evolving set is preserved over time.

\medskip

Motivated by recent results \cite{MugSeiSpa16} and \cite{KimKwo18}, our strategy is to approximate \eqref{model} by the following flow as $\delta\to 0$:
\begin{align}
\label{modeld}
\begin{cases}
V &=-H+\ldt, \quad \ldt := \gd(|\ot|) \hbox{ on } \Gamma_t,\\
\Om^\delta_0 &= \Om_0.
\end{cases}
\end{align}
where $\gd: \R^+ \rightarrow \R$ for $\delta>0$ is defined by
\begin{align}
\label{eqn:gd}
\gd(s) := \frac{1}{\delta} (1 - s).
\end{align}

Let us mention that the comparison principle does not hold for both \eqref{model} and \eqref{modeld}, so  the notion of solutions should be understood as viscosity solutions with a priori given $\ldt$ (see Definition~\ref{def:model} and Definition~\ref{def:viss}). Compared to the original flow \eqref{model}, \eqref{modeld} holds an advantage that $\ldt$ only depends on $|\ot|$, thus it can be handled with little information on the regularity of $\Gamma_t$, which makes it easier to handle with viscosity solutions theory.  The existence and uniqueness for viscosity solutions of \eqref{modeld} were proved in \cite{KimKwo18}. The following is summary of the main results in Theorem~\ref{thm:ap} \& \ref{thm:l2}, Corollary~\ref{thm:exi} and Theorem~\ref{thm:cov}.

\begin{THM}\label{mainthm}
Let us denote $\mathcal{B}(x_0)$ to be the ball of unit volume centered at $x_0$. Under the geometric assumption on the initial data,
\begin{equation}
\label{as:a}
\Omega_0 \hbox{ satisfies } \rho\hbox{-reflection (See Definition~\ref{def:rho}) for some } \rho \in [0, (c_n 5)^{-1}), c_n = |B_1(0)|^{1/n}. 
\end{equation}
There exists a viscosity solution $(\oit, \lambda_{\infty})$ of \eqref{model} approximated by solutions $\{ (\odt, \lambda_{\delta}) \}_{\delta > 0}$ of \eqref{modeld}  as $\delta>0$ with the following properties:
 
 \begin{itemize}
\item[(a)] Along a subsequence $\delta$ and for any finite time $T$, we have
$$
 \max_{0\leq t\leq T} d_H(\odt, \oit) \rightarrow 0,  \ld \weakly \li \hbox { in } L^2([0,T]).
$$
\item[(b)] There exists $r>0$ such that for all $t, \delta>0$ both $\oit$ and $\odt$ contain the ball $B_r(0)$ and stay star-shaped with respect to it.\\
\item[(c)] $\oitt$ uniformly converges to a ball of volume $1$, modulo translation, i.e., 
$$
\sup_{x_0\in \R^n} d_H(\oit, \mathcal{B}(x_0)) \to 0 \hbox{ as } t\to 0.
$$

\end{itemize}
\end{THM}

Let us briefly discuss the main ingredients and challenges in the context of literature.

\medskip

{\bf Geometric properties } Due to the low-dimensional nature of the flow, finite-time singularities even for smooth $\oz$ can be expected in general. It is well known (\cite{Hui87}) that convexity is preserved in the flow \eqref{model}, and the global-time behavior of convex evolution, as well as exponential convergence to the unit ball, has been studied in the smooth case \cite{Hui87} and for anistropic flow \cite{And01} and \cite{BCCN09}. Our goal in this paper is understanding the evolution of star-shaped sets. While it is suspected that star-shapedness is preserved in the evolution, it remains open to be proved. In \cite{KimKwo18} we instead considered a stronger version of star-shapeness, i.e.  the property {\it $\rho$-reflection} given in Section~\ref{sec:not}. Roughly speaking this property amounts to the boundary of the set being Lipschitz with respect to the spherical coordinate given by $B_{\rho}(0)$. \cite{KimKwo18} shows, by moving planes argument, that this property is preserved in the flow with volume-dependent forcing, which includes \eqref{modeld}. In particular this property implies $(b)$ for $\odt$, as well as an equi-continuity over time, yielding the first part of $(a)$. It should be pointed out that, as in \cite{KimKwo18}, our geometric arguments should be incorporated with the variational methods, since the underlying gradient flow structure of \eqref{model} and \eqref{modeld}  provides both existence and asymptotic convergence results for both problems. For this reason our construction of solutions for \eqref{model}-\eqref{modeld} employs constrained minimizing movements with admissible sets only for star-shaped sets, which differs from the standard constructions. 

\bigskip

{\bf Regularity of $\partial\oit$ and Notions of solutions.}
To yield the second part of $(a)$, we obtain uniform $L^2$ bound for $\ld$, largely following the variational arguments in \cite{MugSeiSpa16}, adapted to our constrained minimizing movements described above. The main difficulty that is new in this paper is the lack of the uniform $L^\infty$ bound on $\li$. The bounds for $\ld$ correlates to that of the {\it total curvature} $\int_{\partial\Omega_t} H dS$. An $L^\infty$ bound for $\li$  along with the geometric property of $\Omega_t$ would invoke parabolic regularity theory for curvature flows to yield smoothness of the flow, which in turn yields sufficient compactness to discuss the asymptotic behavior of the flow. Indeed this was the case for \cite{KimKwo18}, where $\ld$ in \eqref{modeld} is a priori bounded by $\frac{1}{\delta}$. 

\medskip

 For convex case, Minkowski's quadratic inequality (See (78) in \cite{BCCN09} and Proposition 4.1 in \cite{And01}) yields a uniform bound on the total mean curvature of a set in terms of volume and perimeter. However, for non-convex set, this inequality fails and the total mean curvature can be unbounded (See Example~\ref{ex:tm1} and Example~\ref{ex:tm2}).

\medskip

For us there is only $L^2$ estimates are available on $\li$, which is inherited from $\ld$'s (see Section~\ref{sec:l2}). For this reason, we fall short of obtaining regularity of $\partial\oit$ that goes beyond Lipschitz. In particular this necessitates a notion of viscosity solutions of $V=-H+\lam$ for a priori given $\lam \in L^1_{loc}([0,\infty))$ (Definition~\ref{def:visl}). Moreover, to assert that the limit $(\oit, \li)$ solves \eqref{model}, our notion needs to stay stable under weak convergence of $\lam$ in $L^2$.  Once such notion is established for prescribed $\lam$, we can introduce a notion of viscosity solutions of \eqref{model}:

\begin{DEF}
\label{def:model} 
The pair $(\ott,\lam)$ be a viscosity solution of \eqref{model} if
$|\Omega_t| = |\Omega_0|$ and $\ott$ is a viscosity solution (See Definition~\ref{def:visls}) of $V=-H+\lam(t)$.
\end{DEF}
The extended notion for prescribed $\lam$, Definition~\ref{def:visl}, enables us to analyze geometric properties of $(\ott,\lam)$ for $\lam \in L^1_{loc}([0,\infty))$. Notions of viscosity solutions for time-integrable operator are previously introduced in \cite{Ish85},  \cite{Bou081} and \cite{Bou082}.  These previous notions however do not allow stability under weak convergence of operators, and thus in this aspect our notion is new. Our notions however coincide with the previous ones as a consequence of its stability properties, see Remark~\ref{coincidence}.

\medskip

Note that gradient and curvature estimates of volume preserving mean curvature and classical mean curvature were proven in \cite{Ath97} for rotationally symmetric case and \cite{EH89} for entire graphs. We expect that this arguments with interior estimates from \cite{EH91} and \cite{Eck04} can be applied for our case with suitable modifications, but we do not pursue this issue here. With higher regularity, uniqueness of the solution for \eqref{model} may be shown by dilation arguments as in \cite{Gig06} and \cite{BCCN09}.

\medskip

{\bf Long-time behavior of the evolution} As mentioned above, we are short of proving smoothness of $\oit$ beyond its Lipschitz graph property, though we expect it to be true. Note that in non-smooth or non-convex setting, perimeter difference may not converge into zero as Hausdorff distance converges to zero. This poses a challenge for proving asymptotic convergnce of $\oit$.  Our proof of perimeter convergence in the asymptotic limit uses both the uniform $L^2$ estimates of mean curvature and star-shapedness (See Lemma~\ref{lem:per}).  \cite{Esc98} and \cite{AKS10} show global well-posedness and exponential convergence if the initial condition is sufficiently close to a round sphere in H\"{o}lder norm and Sobolev norm. respectively. Similar results were proven for sufficiently small traceless second fundamental form of the initial condition in \cite{Li09}. We mention that most of existing results on asymptotic convergence require regularity of the interface to be smoother than $C^{1,\alpha}$.

\medskip

We finish this section with an outline of the paper. In section~\ref{sec:not} we recall  level set formulation of \eqref{model} and notions of the corresponding viscosity solutions for a prescribed and continuous $\lam$. Then we extend the notion to $\lam\in L^1_{loc}([0,\infty))$ and establish its well-posedness by comparison principle. Then we are able to define the notion of solutions for our original flow \eqref{model} as given in Definition~\ref{def:model}.   
In section~\ref{sec:ap} we introduce the approximation by \eqref{modeld} constructed by a constrained minimizing movement. Based on their geometric properties, we establish the first part of Theorem~\ref{mainthm} (a) for the limiting set $\Omega_{\infty}$. Section~\ref{sec:l2} completes the statement of Theorem~\ref{mainthm}(a) and (b)  by establishing a uniform  $L^2$ bound of $\lambda_{\delta}$, using the variational construction of solutions for \eqref{modeld}. This leads to the weak convergence of $\lambda_{\delta}$ to $\lambda_{\infty}$, While following the outline given by \cite{MugSeiSpa16}, our construction of local variation is more delicate (Lemma~\ref{lem:bdd} and Lemma~\ref{lem:dil}), since the perturbed set needs to stay within our geometric constraints. Finally in section~\ref{sec:cov} we prove Theorem~\ref{mainthm}(c), by establishing the perimeter convergence of $\oit$ as $t\to\infty$, using the $L^2$ bound on $\lambda_{\infty}$ obtained in section~\ref{sec:l2}.

\section{Preliminaries and a notion of solution}
\label{sec:not}

\subsection{Notations}
We begin with a list of definitions.
\begin{itemize}
\item $Q:=\R^n \times [0,\infty), \quad Q_T:=\R^n \times [0,T]$;\\
\item $D_r(x_0,t_0) := B_r(x_0) \times (t_0 - r^2,t_0], \quad \partial_p D_r := (\oB_r(x_0) \times \{ t_0 \}) \cup ( \partial B_r(x_0) \times [t_0 - r^2, t_0])$;\\
\item $C_{r,h}(x):= x + [-h,h] \times B_r^{n-1}(0), \,\,   C_{r,h}^{+} (x):= x + [0,h] \times B_r^{n-1}(0), \,\, B_r^{n-1}(0) := \{x\in \R^{n-1}, |x|\leq 1\}$;
\end{itemize}

\medskip

Next we recall some geometric properties from \cite{FelKim14}.

\begin{DEF}\cite[Definition 10]{FelKim14}
\label{def:rho}
A bounded, open set $\Omega$ satisfies $\rho$\textit{-reflection} if 
\begin{enumerate}
\item $\Omega$ contains $\overline{B_\rho(0)}$ and
\item $\Omega$ satisfies that for all direction $\nu \in S^{n-1}$ and all $s>\rho$.
\begin{align}
\Psi_{\Pi_\nu(s)} (\Omega \cap \Pi^+_{\nu}(s)) \subset \Omega \cap \Pi^-_{\nu}(s) 
\end{align}
\end{enumerate}
where $\Psi_{\Pi_\nu(s)}$ is a reflection function with respect to the hyperplane $\Pi_\nu(s):= \{x\cdot \nu =s\}$ defined by
\begin{align*}
\Psi_{\Pi_\nu(s)}(x) := x- 2 \langle x-s\nu, \nu \rangle \nu, \quad \Pi^+_{\nu}(s) : = \{ x \in \R^n : x \cdot \nu > s \} \hbox{ and } \Pi^-_{\nu}(s) : = \{ x \in \R^n : x \cdot \nu < s \}. 
\end{align*}
\end{DEF}

\begin{DEF}\label{def:star}
A bounded set $\Omega$ in $\mathbb{R}^n$ is \textit{star-shaped with respect to a ball} $B_r$ if for any point $y \in B_r$, $\Omega$ is star-shaped with respect to $y$. Let 
\begin{equation}\label{admissible}
S_r:=\{\Omega:\textit{star-shaped with respect to } B_r(0)\} \hbox{ and } S_{r,R}:=S_r\cap\{\Omega:\Omega\subset B_R(0)\}.
\end{equation}
\end{DEF}

\medskip

\subsection{Preliminary notions of viscosity solutions }

For  $u : L \subset \R^{d} \rightarrow \mathbb{R}$ we denote 
its  semi-continuous envelopes $u_* , u^*: \oL \rightarrow \R$ by 
\begin{align}
\label{eqn:env}
u_*(x) := \lim_{\epsilon \downarrow 0} \inf_{|x-y| < \epsilon, \atop y \in L } u(y)\quad \hbox{ and } u^*(x) := \lim_{\epsilon \downarrow 0} \sup_{|x-y| < \epsilon, \atop y \in L} u(y).
\end{align}

\medskip

 For a sequence of functions $\{ u_k \}_{k \in \natural}$ on $Q$, 
\begin{align}
\label{eqn:sups}
\limsups_{k \rightarrow \infty}u_k(x,t) &:= \lim\limits_{j \to \infty} \sup \left\{ u_k(y,s) : k \geq j,\quad |y-x| \leq \frac{1}{j}, \quad |s-t| \leq \frac{1}{j} \right\},\\
\liminfs_{k \rightarrow \infty}u_k(x,t) &:= \lim\limits_{j \to \infty} \inf \left\{ u_k(y,s) : k \geq j,\quad |y-x| \leq \frac{1}{j}, \quad |s-t| \leq \frac{1}{j} \right\}.
\end{align}

\medskip

In the level set formation, $\Omega_t$ is given by $\Omega_t(u):=\{ x \in \R^n : u(x,t) > 0\}$ where $u : Q \rightarrow \R$ solves the following equation:
\begin{align}
\label{main}
u_t = F(Du,D^2 u) + \lt|Du|
\end{align}
where $F: ( \R^n \setminus \{0\})  \times \mathcal{S}^{n \times n} \rightarrow \R$
is given by
\begin{align}
\label{eqn:F}
F(p,X) := \tr\left( \left(I - \dfrac{p}{|p|} \otimes \dfrac{p}{|p|}\right)X \right).
\end{align}
with initial data
\begin{equation}\label{initial}
u(x,0) = u_0(x) := \chi_{\Omega_0} - \chi_{\Omega_0^C} \hbox{ for } x \in \mathbb{R}^n.
\end{equation}

\medskip

Let us recall definitions of viscosity solutions of \eqref{main} and \eqref{modeld} with fixed $\lam \in C(\Rpz)$.

\begin{DEF}
\label{def:vis}
\cite[Definition 2.1]{CCG91}, \cite[Definition 6.1]{Barles:2013gx} 
\begin{itemize}
\item
A function $u: Q\to \mathbb{R}$ is \textit{a viscosity subsolution} of \eqref{main} if  $u^*<\infty$ and for any $\phi \in \C^{2,1}(Q)$ that touches $u^* $ from above at $(x_0,t_0)$ we have
\begin{align}
\label{eqn:vis}
\phi_t(x_0,t_0) \leq  F^*(D\phi(x_0,t_0), D^2\phi(x_0,t_0)) + \lambda(t_0)|D\phi(x_0,t_0)|.
\end{align}
where $F$ is given in \eqref{eqn:F}. 
Also, $u^*$ and $F^*$ are given in \eqref{eqn:env}.
\item 
A function $u: Q\to \mathbb{R}$ is \textit{a viscosity supersolution} of \eqref{main} if $u_*>-\infty$ and for any $\phi \in \C^{2,1}(Q)$ that touches $u_* $ from below at $(x_0,t_0)$ we have
\begin{align*}
\phi_t(x_0,t_0) \geq  F_*(D\phi(x_0,t_0), D^2\phi(x_0,t_0)) + \lambda(t_0)|D\phi(x_0,t_0)|.
\end{align*}
where $F$ is given in \eqref{eqn:F}. Also, $u_*$ and $F_*$ are given in \eqref{eqn:env}.
\item 
A function $u: Q\to \mathbb{R}$ is \textit{a viscosity solution} of \eqref{main}-\eqref{initial} (or \eqref{initial_c})
if $u^*$ is \textit{a viscosity subsolution} of \eqref{main} and $u_*$ is \textit{a viscosity supersolution} of \eqref{main}, and if $u^*= {(u_0)}^*$ and $u_*= {(u_0)}_*$ at $t=0$. 
\item
For any $\lam \in C(\Rpz)$, 
$\ott$ be \textit{a viscosity solution (subsolution or supersolution, respectively) } of 
\begin{align}
\label{modelg}
V=-H + \lam(t)
\end{align}
if $u:= \chi_{\ot} - \chi_{\ot^c}$ is \textit{a viscosity solution (subsolution or supersolution, respectively)} of \eqref{main}-\eqref{initial}.
\end{itemize}
\end{DEF}

\medskip

We also recall definitions of classical solutions and equivalent definitions of viscosity solutions of \eqref{main} with fixed $\lam \in C(\Rpz)$.

\begin{DEF}
\label{def:ssol}
$\,$
Consider a cylinder $D_r \subset Q$
and $F$ given in \eqref{eqn:F}.
\begin{itemize}
\item
A function $\phi \in C^{2,1}(D_r)$ is \textit{a classical subsolution} in $D_r$ of \eqref{main} if it holds that
\begin{align}
\label{eqn:1ssol}
\phi_t \leq F_*(D\phi,D^2\phi) + \lambda |D\phi| \hbox{ in } D_r.
\end{align}
\item
A function $\phi \in C^{2,1}(D_r)$ is \textit{a classical supersolution} in $D_r$ of \eqref{main} if it holds that
\begin{align}
\label{eqn:2ssol}
\phi_t \geq F^*(D\phi,D^2\phi) + \lambda |D\phi| \hbox{ in } D_r.
\end{align}
\item
We say that $\phi \in C^{2,1}(D_r)$ is  \textit{a classical strict subsolution (supersolution, respectively)} on $D_r$ of \eqref{main} if the strict inequality of \eqref{eqn:1ssol} (\eqref{eqn:2ssol}, respectively) holds in $D_r$
\end{itemize}
\end{DEF}

\reservespace{
\begin{DEF}
\label{def:visd}
\cite[Definition 7.2]{CafSal05}
\begin{itemize}
\item
A function $u: Q\to \mathbb{R}$ is \textit{a viscosity subsolution} of \eqref{main} if  $u^*<\infty$ and for $D_r \subset Q$ and for every classical strict supersolution $\phi \in C^{2,1}(D_r)$, $u^* < \phi$ on $\partial_p D_r$ implies $u^* < \phi$ in $\overline{Q}$.
\item A function $u: Q\to \mathbb{R}$ is \textit{a viscosity supersolution} of \eqref{main} if $u_*>-\infty$ and $D_r \subset Q$ and for every classical strict subsolution $\phi \in C^{2,1}(D_r)$, $u_* > \phi$ on $\partial_p D_r$ implies $u_* > \phi$ in $\overline{Q}$.
\end{itemize}
\end{DEF}
}

\begin{DEF}
\label{def:viss} 
\cite[Definition 2.7]{KimKwo18}
For $\lam \in C(\Rpz)$,
$(\ott,\lam)$ be a viscosity solution of 
\begin{align}
\label{modelgam}
V=-H + \gamma(|\ot|)
\end{align}
where $\gamma : \Rpz \rightarrow \R$ if  $\ott$ is a viscosity solution of \eqref{modelg} with $\gamma(|\ot|)=\lam(t)$.
\end{DEF}

\medskip

\subsection{Viscosity Solutions for $L^1_{loc}$ forcing}
In this section, we develop a notion of viscosity solutions for \eqref{main} for a fixed $\lam$ in $L^1_{loc}([0,\infty))$. Some notations are in order.
For $\gam \in C(\Rpz)$, the sup convolution $\hu(\cdot;\gam)$ and inf convolution $\tu(\cdot;\gam)$ is given by
\begin{align}
\label{eqn:sup}
\hu(x,t; \gam) &:= \sup_{y \in \overline{B}_{\gam(t)}(x)} u(y,t),\\ 
\label{eqn:inf}
\tu(x,t; \gam) &:= \inf_{y \in \overline{B}_{\gam(t)}(x)} u(y,t). 
\end{align}
Note that $\widehat{u^*} = (\widehat{u})^*$ and $\widetilde{u_*}  = (\widetilde{u})_*$.

\begin{DEF} 
\label{def:visl} 
For $\lam \in L^1_{loc}([0,\infty))$, $\Lambda(t):= \int_0^t  \lam(s) ds$ and $F$ given in \eqref{eqn:F},
\begin{itemize}
\item
A function $u: Q \to \mathbb{R}$ is \textit{a viscosity subsolution} of \eqref{main} if  $u^*<\infty$ and for any $0 \leq t_1 < t_2$ and $ \Theta \in \C^1((t_1,t_2)) \cap \C([t_1,t_2])$ such that $\Theta \geq \Lambda$ in $[t_1,t_2]$, a function $\hu = \hu(\cdot; \Theta - \Lam)$ given in \eqref{eqn:sup} is a viscosity subsolution of
\begin{align}
\label{eqn:the}
u_t =  F(Du,D^2 u) + \Theta' |Du|
\end{align}
in $(t_1, t_2) \times \R^n$ in the sense of Definition~\ref{def:vis}.
\item
A function $u: Q \to \mathbb{R}$ is \textit{a viscosity supersolution} of \eqref{main} if  $u_*> -\infty$ and for any $0 \leq t_1 < t_2$ and $\Theta \in \C^1((t_1,t_2)) \cap \C([t_1,t_2])$ such that $\Theta \leq \Lambda$ in $[t_1,t_2]$, a function $\tu = \tu(\cdot; -\Theta+\Lam)$ given in \eqref{eqn:inf} is a viscosity supersolution of \eqref{eqn:the}
in $(t_1, t_2) \times \R^n$ in the sense of Definition~\ref{def:vis}.
\item 
A function $u: Q\to \mathbb{R}$ is \textit{a viscosity solution} of \eqref{main}-\eqref{initial} (or \eqref{initial_c})
if $u^*$ is \textit{a viscosity subsolution} of \eqref{main} and $u_*$ is \textit{a viscosity supersolution} of \eqref{main}, and if $u^*= {(u_0)}^*$ and $u_*= {(u_0)}_*$ at $t=0$. 
\end{itemize}
\end{DEF}

Note that if $\lam$ is continuous, then our definition coincides the usual definition. We also define the corresponding notion of viscosity solutions for sets.

\begin{DEF}
\label{def:visls}
For $\lam \in L^1_{loc}([0,\infty))$,
let $\ott$ be \textit{a viscosity solution (subsolution or supersolution, respectively)} of $V=-H+\lam(t)$ if $u:= \chi_{\ot} - \chi_{\ot^c}$ is \textit{a viscosity solution (subsolution or supersolution, respectively)} of \eqref{main}-\eqref{initial}.
\end{DEF}

Recall that the definition of solutions for \eqref{model}, is based on above definition.

\begin{repDEF}{def:model}
For $\lam\in L^1_{loc}([0,\infty))$, the pair $(\ott,\lam)$ be a viscosity solution of \eqref{model} if
$|\Omega_t| = |\Omega_0|$ and $\ott$ is a viscosity solution of $V=-H+\lam(t)$.
\end{repDEF}

\begin{rem}
\label{rem:coi}
Note that for $\lambda \in \C(\Rpz)$, this definition coincides with Definition~\ref{def:vis}. First of all, \cite[Lemma 2.5]{KimKwo18} implies that a viscosity subsolution \resp in the sense of Definition~\ref{def:vis} is that in the sense of Definition~\ref{def:visl}. On the other hand, if $\lambda \in \C(\Rpz)$, then $\Lam \in \C^1(\Rpz)$. Thus, we can choose $\Theta = \Lam$. As $\hu(\cdot ; 0) = \tu(\cdot ; 0) = u$ in $Q$, we conclude that a viscosity subsolution \resp in the sense of Definition~\ref{def:visl} is that in the sense of Definition~\ref{def:vis}.
\end{rem}

In the rest of this section, we develop existence and uniqueness results for \eqref{main}. We first show the comparison principle in Theorem~\ref{thm:comp}, which yields uniqueness (Corollary~\ref{cor:vuni}). Moreover, we show the stability of viscosity solutions of $V=-H+\lam_k(t)$ for $\{ \lam_k \}_{ k \in \natural } \subset L^1_{loc}([0,\infty))$ when a sequence of time integrals of $\lam_k$ converges. This yields existence (Corollary~\ref{cor:vexd}).

\reservespace{
\begin{thm}
\label{thm:comp}
Let $u$ and $v$ be a viscosity subsolution and supersolution of \eqref{main}, respectively, in the sense of Definition~\ref{def:visl}.
If for some $r>0$ and $(x_0,t_0) \in Q$ we have $u^* \leq v_* \hbox{ on } \partial_p D_r(x_0,t_0)$, then
\begin{align}
\label{eqn:2comp}
u^* \leq v_* \text{ on } D_r(x_0,t_0).
\end{align}
\end{thm}
}

\begin{proof}
For simplicity, consider $(x_0,t_0)=(0,r^2)$ and denote $D_r:= D_r(0,r^2) = B_r(0)\times (0,r^2]$. Note that  we may assume the following, by adding a small constant to $v$: 
\begin{align}
\label{eqn:11comp}
u^* < v_* \hbox{ on } \partial_p D_r(x_0,t_0).
\end{align}

\medskip

{ 1.} Let us show that there exists $\e_1>0$ such that
\begin{align}
\label{eqn:comp11}
\hu^*(\cdot; \e_1) < \tv_*(\cdot; \e_1) \hbox{ on } \partial_p D_r.
\end{align}
Suppose that \eqref{eqn:comp11} does not hold for all $\e_1>0$. Then, there exists a sequence $\{ x_k \}_{k \in \natural} \subset \partial_p D_r$ such that 
\begin{align}
\hu^*\left(x_k; \frac{1}{k}\right) \geq \tv_*\left(x_k;  \frac{1}{k}\right).
\end{align}
By the semi-continuity of $u^*$ and $v_*$, there exists $\{ (y_k,z_k) \}_{k \in \natural}$ such that
\begin{align}
\label{eqn:comp12}
|x_k - y_k| \leq \frac{1}{k}, \quad |x_k - z_k| \leq \frac{1}{k}
\end{align}
and
\begin{align}
\label{eqn:comp13}
u^*(y_k) \geq v_*(z_k)
\end{align}

\medskip

By compactness of $\oD_{r+1}$, there exists a subsequence $\{ k_i \}_{i \in \natural}$ and $(y^*,z^*)$ such that $\{ (y_{k_i},z_{k_i}) \}_{k_i \in \natural}$ converges to $(y^*,z^*)$. From \eqref{eqn:comp12} and the closedness of $\pdr$, we conclude that $y^* = z^* \in \pdr$. From \eqref{eqn:comp13} and the semi-continuity of $u^*$ and $v_*$, it holds that
\begin{align}
\label{eqn:comp14}
u^*(y^*) \geq \limsup_{i \to \infty} u^*(y_{k_i}) \geq \liminf_{i \to \infty} v_*(z_{k_i}) \geq v_*(z^*) = v_*(y^*)
\end{align}
This contradicts to \eqref{eqn:11comp}.

\medskip

{ 2.} Note that $\C^1([0,r^2])$ is dense in $\C([0,r^2])$. There exists $\Theta \in \C^1([0,r^2])$ such that
\begin{align}
\label{eqn:comp21}
\sup_{ t \in [0,r^2] } | \Lam(t) - \Theta(t) | \leq \frac{\e_1}{2}.
\end{align}
where $\e_1>0$ is given in Step 1. Then, $\hu^*(\cdot; \frac{\e_1}{2} + \Theta(t) - \Lam(t))$ and $\tv_*(\cdot; \frac{\e_1}{2} - \Theta(t) + \Lam(t))$ are well-defined in $D_r$. Note that $\hu^*$ and $\tv_*$ given above are respectively viscosity subsolution and supersolutions of \eqref{eqn:the}. 

\medskip

From \eqref{eqn:comp21} and \eqref{eqn:comp11}, it holds that
\begin{align}
\hu^*\left(\cdot; \frac{\e_1}{2} + \Theta(t) - \Lam(t)\right) \leq \hu^*(\cdot; \e_1) < \tv_*(\cdot; \e_1) \leq  \tv_*\left(\cdot; \frac{\e_1}{2} - \Theta(t) + \Lam(t)\right).
\end{align}
on $\pdr$. From comparison principle for \eqref{eqn:the} in \cite[Theorem 4.1]{CCG91}, we conclude that
\begin{align}
\hu^*\left(\cdot; \frac{\e_1}{2} + \Theta(t) - \Lam(t)\right) <  \tv_*\left(\cdot; \frac{\e_1}{2} - \Theta(t) + \Lam(t)\right).
\end{align}
on $D_r$, which implies \eqref{eqn:2comp}.
\end{proof}

\medskip

\begin{cor}
\label{cor:vuni}
For $g \in \C(\pdr)$, there is at most one viscosity solution $u$ of \eqref{main} with $u^* = u_*=g$ on $\pdr$ in the sense of Definition~\ref{def:visl}.
\end{cor}

\medskip

Next we develop stability results for $\{ \lambda_k \}_{k \in \natural}$ such that $\{ \Lambda_k \}_{k \in \natural}$ uniformly converges to $\Li$
where
\begin{align}
\label{eqn:Lk}
\{ \lam_k \}_{ k \in \natural \cup \{ + \infty \} } \subset L^1_{loc}([0,\infty)) \hbox{ and } \Lambda_k(t):=\int_{0}^{t}\lambda_k(s)ds \hbox{ for } k \in \natural \cup \{ + \infty \}. 
\end{align}
Note that this gives stability results for weak convergence of $\{ \lambda_k \}_{k \in \natural} \subset L^{p}(\Rpt)$ for any $p \in (1,\infty]$. This results will be used Corollary~\ref{thm:exi}.

\medskip

\begin{thm}
\label{thm:stf}
For $\{ \lam_k \}_{ k \in \natural \cup \{ + \infty \}}$ and $\{ \Lk \}_{k \in \natural \cup \{ + \infty \}}$ given in \eqref{eqn:Lk}, assume that $\{ \Lk \}_{k \in \natural }$ locally uniformly converges to $\Li$. Let $\{ u_k \}_{k \in \natural}$ be a sequence of viscosity subsolutions (supersolutions, respectively) of \eqref{main} with $\lambda = \lambda_k$ for all $k \in \natural$. If $u:= \limsups_{k \rightarrow \infty} u_k<\infty$ ( $u:=\liminfs_{k \rightarrow \infty}u_k > -\infty$, respectively), then $u$ is a viscosity subsolution (supersolutions) of \eqref{main} with $\lam = \li$ in the sense of Definition~\ref{def:visl}.
\end{thm}

\begin{proof}
We only show the subsolution part, since the rest can be shown with parallel arguments. Let $\{ u_k \}_{k \in \natural}$ be a sequence of viscosity subsolutions. 
\medskip

{ 1.} Choose any $0 \leq t_1 < t_2$ and $ \Theta \in \C^1((t_1,t_2)) \cap \C([t_1,t_2])$ such that $\Theta \geq \Li$ in $[t_1,t_2]$. Let us show that $\hu(\cdot;\Theta - \Li)$ given in \eqref{eqn:sup} is a viscosity subsolution of \eqref{eqn:the}. From the equivalent definition of viscosity solutions in Definition~\ref{def:visd}, it is enough to show that for any $D_r \subset Q$
\begin{align}
\label{eqn:stf10}
\hu^*(\cdot;\Theta - \Li) < \phi \hbox{ in } D_r
\end{align}
where $\phi \in \C^{2,1}(D_r)$ is a classical strict supersolution of \eqref{eqn:the} given in Definition~\ref{def:ssol} such that 
\begin{align}
\hu^*(\cdot;\Theta - \Li) < \phi \hbox{ on } \partial_p D_r.  
\end{align}
First, as $u < + \infty$ and $u$ is upper semicontinuous, we get $\hu^* < \infty$. Next, by the upper semicontinuity of $u^*$, there exists $\e_2>0$ such that 
\begin{align}
\hu^*(\cdot;\Theta - \Li) < \phi - 3 \e_2 \quad \hbox{ on } \partial_p D_r.
\end{align}
From the upper semicontinuity again, there exists $\e_1>0$ such that 
\begin{align}
\label{eqn:stf20}
\hu^*(\cdot;\e_1 + \Theta - \Li) < \phi - 2 \e_2 \quad \hbox{ on } \partial_p D_r.
\end{align}
By uniform convergence of $\Lk$, there exists $k_1 \in \natural$ such that for all $k > k_1$, it holds that 
\begin{align}
\label{eqn:stf11}
|\Li - \Lk| < \frac{\e_1}{2}
\end{align}
in $[t_1,t_2]$.
By definition, $\hu_k=\hu_k(\cdot ;\e_1 + \Theta - \Lk)$
is a viscosity subsolutions of \eqref{eqn:the} in $(t_1,t_2)$ for all $k > k_1$.

\medskip

{ 2.}  Let us show that there exists $k_2 \in \natural$ such that $k_2 > k_1$ and
\begin{align}
\label{eqn:stf21}
\hu_{k}^*(\cdot ;\e_1 + \Theta - \Lk) < \phi - \e_2 \hbox{ on } \partial_p D_r \hbox{ for all } k \geq k_2.
\end{align}
where $k_1$ is given in Step 1. Suppose that such $k_2$ does not exist. 
Then, there exists a sequence $\{ k_i \}_{i \in \natural}$ converging to infinity and $\{ x_{k_i} \}_{i \in \natural} \subset \pdr$ such that $k_i \geq k_1$ and
\begin{align}
\label{eqn:stf22}
\hu_{{k_i}}^*(x_{k_i};\e_1 + \Theta - \Lki) \geq \phi(x_{k_i}) - \e_2 \hbox{ on } \partial_p D_r \hbox{ for all } i \in \natural.
\end{align}

\medskip

By the upper semi-continuity of $u^*$, there exists $\{ y_{{k_i}} \}_{i \in \natural}$ such that
\begin{align}
\label{eqn:stf221}
|y_{k_i} - x_{k_i} | \leq \e_1 + \Theta - \Lki \quad \hbox{ and } u_{{k_i}}^*(y_{k_i}) \geq \phi(x_{k_i}) - \e_2 
\end{align}
Furthermore, there exists $\{ z_{{k_i}} \}_{i \in \natural}$ such that
\begin{align}
\label{eqn:stf25}
|z_{k_i} - y_{k_i} | \leq \frac{1}{{k_i}} \quad \hbox{ and }
u_{{k_i}}(z_{k_i}) + \e_2 \geq u_{{k_i}}^*(y_{k_i}) 
\end{align}
From \eqref{eqn:stf221} and \eqref{eqn:stf25}, we get
\begin{align}
\label{eqn:stf24}
|z_{k_i} - x_{k_i} | \leq \e_1 + \Theta - \Lki + \frac{1}{{k_i}} \hbox{ and } u_{{k_i}}(z_{k_i}) \geq \phi(x_{k_i}) - 2\e_2
\end{align}
As $\{ x_{{k_i}} \}_{i \in \natural} \subset \pdr$, \eqref{eqn:stf11} and \eqref{eqn:stf24} imply that 
\begin{align}
\{ z_{{k_i}} \}_{i \in \natural} \subset \oD_{\hr} \hbox{ where } \hr= r+2\e_1+\| \Theta - \Li \|_\infty + 1.
\end{align}

\medskip

From compactness of $\oD_{\hr}$, there exists a subsequence $\{ k_{i_j} \}_{i \in \natural}$ and $(x^*,z^*)$ such that $\{ (x_{k_{i_j}},z_{k_{i_j}}) \}_{j \in \natural}$ converges to $(x^*,z^*)$.  \eqref{eqn:stf24} implies that 
\begin{align}
|z^* - x^* | \leq \e_1 + \Theta - \Li 
\end{align}
and
\begin{align}
\label{eqn:stf26}
u(z^*) \geq \limsup_{j \to \infty} u_{k_{i_j}}(z_{k_{i_j}}) \geq \limsup_{j \to \infty} \phi(x_{k_{i_j}}) - 2\e_2 = \phi(x^*) - 2\e_2
\end{align}
This contradicts to \eqref{eqn:stf20} and we conclude \eqref{eqn:stf21}.

\medskip

{ 3.} 
From Step 1 and \eqref{eqn:stf21}, comparison principle in Theorem~\ref{thm:comp} implies that
\begin{align}
\hu_{k}^*(\cdot ;\e_1 + \Theta - \Lk) < \phi -\e_2 \hbox{ in } D_r \hbox{ for all } k \geq k_2.
\end{align}
where $\e_1$ and $\e_2$ are given in \eqref{eqn:stf20}, and $k_2$ is given in \eqref{eqn:stf21}. The above and \eqref{eqn:stf11} imply that
\begin{align}
u_{k}(y) < \phi(x) -\e_2 \hbox{ for all } x \in D_r \hbox{ and } y \in B_{\frac{\e_1}{2} + \Theta - \Li}(x)  \hbox{ for all } k \geq k_2.
\end{align}
and we conclude \eqref{eqn:stf10}.
\end{proof}

Let us construct radial barriers of \eqref{main}.

\begin{lem}
\label{lem:rb}
For $\Lam : \Rpz \to \R$ given by
\begin{align}
\Lambda(t):= \int_0^t \lambda(s) ds
\end{align}
and $c \in \R$,  define $\z^-:Q \rightarrow \R$ and $\z^+ : Q \rightarrow \R$ by
\begin{align}
\label{eqn:1rb}
\z^-(x,t;\Lam,c) &:= -\chi_{ \{ x \in \R^n : |x| < c-\Lam(t) \} }(x)  \hbox{ and } \quad \z^+(x,t;\Lam,c):= \chi_{ \{ x \in \R^n : |x| < c+\Lam(t) \} }(x).
\end{align}
Then, $\z^-$ and $\z^+$ are a viscosity subsolution and supersolution, respectively, of \eqref{main} in the sense of Definition~\ref{def:visl}.
\end{lem}

\begin{proof}

Let us show that $\z^-$ is a viscosity subsolution of \eqref{main} only. The respective one can be shown by parallel arguments.

\medskip

Choose any $0 \leq t_1 < t_2$ and $ \Theta \in \C^1((t_1,t_2)) \cap \C([t_1,t_2])$ such that $\Theta \geq \Lambda$ in $[t_1,t_2]$. Let us show that $\hzm(\cdot;\Theta - \Lam)$ given in \eqref{eqn:sup} is a viscosity subsolution of \eqref{eqn:the}.
Note that we have 
\begin{align}
\label{eqn:rb10}
\hzm(x,t ; \Theta - \Lam) =  -\chi_{ \N_t }(x) \hbox{ where } \N_t:= \{ x \in \R^n : |x| < c-\Theta(t) \}
\end{align}
in $Q$.

\medskip

Suppose that $\phi \in \C^{2,1}(Q)$ touches $\hzm$ from above at $(x_0,t_0)$. First, consider the case $|x_0| \neq c - \Theta(t_0)$. In this case, as $\N_t$ given in \eqref{eqn:rb10} moves continuously in time, $\hzm$ is constant near $(x_0,t_0)$. Thus, it holds that
\begin{align}
\label{eqn:rb101}
\phi_t(x_0,t_0) \leq 0, \quad D\phi(x_0,t_0) = 0, \hbox{ and } D^2\phi(x_0,t_0) \geq 0.
\end{align}
The ellipticity of $F$ given in \eqref{eqn:F} and \eqref{eqn:rb101} implies
\begin{align}
\label{eqn:rb11}
\phi_t(x_0,t_0) \leq  F^*(D\phi(x_0,t_0), D^2\phi(x_0,t_0)) + \Theta'(t_0)|D\phi(x_0,t_0)|.
\end{align}

\medskip

Let us consider the case $|x_0| = c - \Theta(t_0)$. If either $x_0=0$ or $x_0$ is a local minimum of $\phi(\cdot, t_0)$, then by the parallel arguments above, we get \eqref{eqn:rb101} and \eqref{eqn:rb11}. Otherwise, both $\N_t$ given in \eqref{eqn:rb10} and a sublevel set $O_t$ of $\phi$ defined by 
\begin{align}
\label{eqn:rb12}
O_t := \{ x \in \R^n : \phi(x,t) < \phi(x_0,t_0) \}
\end{align}
are nonempty near $(x_0,t_0)$. By comparing the normal velocity and mean curvature of the level sets $\N_t$ and $O_t$, we conclude that
\begin{align}
\frac{\phi_t}{|D\phi|}(x_0,t_0) \leq \Theta'(t_0) \hbox{ and } \nabla \cdot \left(\dfrac{D\phi}{|D\phi|} \right)(x_0,t_0) \geq \frac{n-1}{|x_0|}>0.
\end{align}
which implies \eqref{eqn:rb11}.
\end{proof}

Let us recall $\C_{a}$ from \cite{CCG91} for $\N \subset \R^k$, $k \in \natural$ and $a \in \R$,
\begin{align}
\label{eqn:ca}
\C_{a} (\N):= \{g \in \C (\N) : g-a \hbox{ has compact support in } \N \} 
\end{align}
and consider continuous initial data $g \in \C_a(\R^n)$,
\begin{align}
\label{initial_c}
u(x,0) = u_0(x) := g(x) \hbox{ for } x \in \mathbb{R}^n.
\end{align}
such that $\{ x \in \R^n : g(x) > 0 \} = \oz$ and $\{ x \in \R^n : g(x) < 0 \} = (\oz)^C$.

\medskip

\reservespace{
From Theorem~\ref{thm:comp} and Theorem~\ref{thm:stf} combining with radial barriers in Lemma~\ref{lem:rb}, we get existence and uniqueness of \eqref{main} with continuous initial data.
\begin{thm}
\label{thm:vex}
For $T>0$, there is a unique viscosity solution $u$ in $\C_a(Q_T)$ of \eqref{main}-\eqref{initial_c} in the sense of Definition~\ref{def:visl}. 
\end{thm}
}

\begin{proof} 
As $\C^1([0,T])$ is dense in $\C([0,T])$
, there exists $\{ \Theta_k  \}_{k \in \natural} \subset \C^1([0,T])$ such that $\{ \Theta_k \}_{k \in \natural}$ uniformly converges to $\Lam$ in $[0,T]$. From the existence of viscosity solutions in \cite[Theorem 6.8]{CCG91} of
\begin{align}
\label{eqn:vex11}
u_t =  F(Du,D^2 u) + (\Theta_k)' |Du| &\hbox{ in } Q
\end{align}
with initial data \eqref{initial_c},
there exists a sequence of viscosity solutions $\{ u_k \}_{k \in \natural} \subset \C_a(Q_T)$ of \eqref{eqn:vex11}-\eqref{initial_c}. Here, $F$ and $\C_a$ are given in \eqref{eqn:F} and \eqref{eqn:ca}, respectively.

\medskip

Define
$u^+ := \limsups_{k \rightarrow \infty}u_k \hbox{ and } u^- := \liminfs_{k \rightarrow \infty} u_k$.
As $g \in \C_a(\R^n)$, Theorem~\ref{thm:comp} implies $\|u_k \|_{L^\infty} \leq \| g \|_{L^\infty}$ and thus $\|u^\pm \|_{L^\infty} < +\infty$. Furthermore, by comparing $\{ u_k \}_{k \in \natural}$ with radial barriers $\| g \|_{L^\infty} \z^\pm(\cdot;\Theta_k,c)$ given in Lemma~\ref{lem:rb}  for sufficiently large $c>0$,
we conclude that the supports of $\{ u_k \}_{k \in \natural}$ are uniformly bounded in $Q_T$ for all $k \in \natural$. Thus, $u^\pm - a$ have compactly supports in $Q_T$. 

\medskip

Let us show that
\begin{align}
\label{eqn:vex20}
u^+ = u^- \quad \hbox{ in } Q_T.
\end{align}
First, by definition of $\limsups$ and $\liminfs$ in \eqref{eqn:sups}, it holds that
\begin{align}
\label{eqn:vex13}
u^+ \geq u^- \quad \hbox{ in } Q_T.
\end{align}

\medskip

On the other hand, from the uniform convergence of $\{ \Theta_k \}_{k \in \natural}$ and Theorem~\ref{thm:stf}, $u^+$ and $u^-$ are a viscosity subsolution and supersolution of \eqref{main}-\eqref{initial_c}, respectively. As $(u^+)^* = (u^-)_* = g$ at $t=0$, Theorem~\ref{thm:comp} implies 
\begin{align}
\label{eqn:vex21}
u^+ \leq u^- \quad \hbox{ in } Q_T.
\end{align}
Therefore, we get \eqref{eqn:vex20} from \eqref{eqn:vex13} and \eqref{eqn:vex21}. From Corollary~\ref{cor:vuni}, we conclude that $u^+ (= u^-) $ is a unique viscosity solution in $\C_a(\R^n)$ of \eqref{main}-\eqref{initial_c}. 
\end{proof}

From parallel arguments in \cite[Theorem 2.1]{Barles:1993gaba}, we conclude existence of \eqref{main}-\eqref{initial}.
\begin{cor}
\label{cor:vexd}
There exists a unique viscosity solution of \eqref{main}-\eqref{initial} in the sense of Definition~\ref{def:visl}. 
\end{cor}

\begin{rem}\label{coincidence}
As a consequence of Theorem~\ref{thm:stf} and Theorem~\ref{thm:vex}, we conclude that our notion in Definition~\ref{def:visl} coincides with viscosity solutions in \cite{Bou081}. This can be shown by smooth approximations of the operator and stability of each notions under strong $L^1$-convergence of forcing term.
\end{rem}

\medskip

\section{Approximation of Volume Preserving Mean Curvature Flow}

\label{sec:ap}

In this section, we construct a solution of \eqref{main} by \eqref{modeld}. We show that viscosity solutions $(\odt)_{t>0}$ of \eqref{modeld} are equicontinuous in Hausdorff distance, based on the geometric properties of $\odt$.  This yields uniform convergence of $\odt$ in Hausdorff distance in Theorem~\ref{thm:ap}. We will conclude in Section~\ref{sec:l2} that the limit of this strong convergence is a viscosity solution of \eqref{model}. While much of the results in this section follows that of \cite{KimKwo18}, our focus here to obtain uniform estimates that stay independent of $\delta>0$ as $\delta\to 0$.

\medskip

Here is main theorem of this section.

\begin{thm}
\label{thm:ap}
There exists a sequence $\{ \delta_i \}_{i \in \natural}$ such that  $\delta_i \rightarrow 0$ as $i \rightarrow \infty$ and
\begin{align}
d_H(\Omega^{\delta_{i}}_t, \oit) \rightarrow 0
\end{align}
for some $(\oit)_{t\geq0} \subset S_{r_1,R_1}$ locally uniformly in time as $i$ goes to infinity. As a consequence, $|\oit|=1$ for all $t>0$. Here $\odtt$ is a unique solution of \eqref{modeld} given in Proposition~\ref{thm:geo}.
\end{thm}

Let us briefly explain the outline of proof. Based on \cite{KimKwo18}, we first show that for a small $\delta$ \eqref{modeld} is well-posed and $\odt$ is star-shaped with respect to a ball (See Definition~\ref{def:star}) in Proposition~\ref{thm:geo}. In Proposition~\ref{lem:ht}, based on geometric properties in Lemma~\ref{lem:ineq}, we show that $\odt$ is h\"{o}lder continuous with respect to time. Then, by the equicontinuity of $(\odt)_{t\geq0}$ with respect to both time and space, there exists a converging subsequence. 

\medskip

\reservespace{
\begin{prop} 
\label{thm:geo}
Let $\delta \in (0,\delta_0)$ for $\delta_0$ given in \eqref{eqn:d0}.
\begin{enumerate}
\item
There exists $r_1= r_1(\oz)$ and $R_1=R_1(\oz)>0$ such that $\odt \in S_{r_1,R_1}$ for all $t\geq0$.
\item
There exists a unique viscosity solution $((\odt)_{t\geq0},\ld)$ of \eqref{modeld} that is bounded and has smooth boundaries. 
\end{enumerate}
\end{prop}
}

\begin{proof}
First, let us show that $\gd$ given in \eqref{eqn:gd} satisfies \cite[Assumption A]{KimKwo18} for all $\delta \in \left( 0, \delta_0\right)$.
As $\gd$ is a decreasing function, it holds that
\begin{align}
\gd(|\Omega|) \geq \gd(|B_{5\rho}|) > \frac{n-1}{\rho} 
\hbox{ for all } \Omega \subset \overline{B}_{5\rho} \hbox{ and all } \delta \in \left( 0, \delta_0\right)
\end{align}
where
\begin{align}
\label{eqn:d0}
\delta_0:= \frac{ \rho(1 - |B_{5\rho}|)}{n-1}
\end{align}
Note that $1 - |B_{5\rho}| > 0$ from \eqref{as:a}, and thus we get $\delta_0>0$. On the other hand, $\gd$ is Lipschitz continuous and satisfies 
\begin{align}
\label{eqn:geo1}
\limsup\limits_{R\rightarrow \infty}
 \frac{\gd(|B_R|)}{R} = -\infty < \infty. 
\end{align}
Thus, we conclude that $\gd$ satisfies \cite[Assumption A]{KimKwo18} for all $\delta \in \left( 0, \delta_0\right)$.

\medskip

From Theorem 1 and 2 in \cite{KimKwo18} this problem is well-posed and $\odtt$ satisfies $\rr$ for all $\delta \in \left( 0, \delta_0\right)$. Furthermore, \cite[Equation (3.11)]{KimKwo18} implies that $\odtt \subset S_{r_1}$ where
 $r_1=r_1(\oz)$ is given by
\begin{align}
\label{eqn:r1}
r_1 := \rho(\beta_1^2 + 2\beta_1)^{\frac{1}{2}}
\end{align}
for some $\beta_1 >0$ such that $\overline{B}_{(1+\beta_1) \rho} \subset \oz$. On the other hand due to Lemma~\ref{lem:4rho}, $\oz \subset \subset B_{R_1}$ and $\ldt < 0$ if $\sup_{x \in \odt } |x| \geq R_1$, where
\begin{align}
\label{eqn:r1}
R_1 := 5\rho + w_n^{\frac{1}{n}} \hbox{ and } \,\, w_n := |B_1(0)|.
\end{align} 
A barrier argument with $B_{R_1}$ yields that $\odt \subset B_{R_1}$ for all $t>0$ and all $\delta \in \left( 0, \delta_0\right)$.
\end{proof}

The following discrete time scheme is  a simplified version of Definition 5.1 in \cite{KimKwo18}.

\begin{DEF}
\label{def:mmov}
$\,$
\begin{itemize}
\item
\textit{The one-step discrete gradient flow with a time step $h>0$}, $T = T(\cdot ; h, \delta) \subset \R^n$, 
is defined by
\begin{align}
\label{eqn:Jd}
T(E; h, \delta) \in \argmin_{F \in S_{r_0,R_0}} \cJ_\delta(F) + \frac{1}{h}\tilde{d}^2(F,E),\quad \cJ_\delta(\Om) : = \per(\Om) + \frac{1}{2\delta} (1 - |\Om|)^2,
\end{align}
where \textit{pseudo-distance} $\tilde{d}$ 
is given by
\begin{align}\label{def:td}
\tilde{d}(F,E) := \left({\int_{E \triangle F} d(x,\partial E) dx}\right)^{\frac{1}{2}}, 
\end{align}
Here, $r_0$ and $R_0$ are constants such that 
\begin{align}
\label{eqn:r0}
r_0 \in (0,r_1) \hbox{ and } R_0 > R_1 
\end{align}
for $r_1$ and $R_1$ given in Proposition~\ref{thm:geo}
\item
\textit{The discrete gradient flow with a time step $h>0$ and the initial set $E_0$}, $E_t = E_t(h, \delta) \subset \R^n$, can be defined by for $t\in \Rpz$ 
\begin{align}
\label{eqn:mmov}
E_t = E_t(h, \delta) := T^{[t/h]}(E_0;h,\delta).
\end{align}
Here, $T^{m}$ for $m\in \mathbb{N}$ is the $m$th functional power. 
\end{itemize}
\end{DEF}

Now, we show that $(\odt)_{t\geq0}$ can be approximated locally uniformly by above discrete flow. In Lemma~\ref{lem:ss}, we get short-time star-shapedness based on H\"{o}lder continuity of $\odt$ in time. We postpone the proof into Appendix~\ref{ap:ap} as other arguments are parallel to \cite[Theorem 6.8]{KimKwo18}.

\begin{prop}
\label{prop:nom} Let $\delta \in (0,\delta_0)$ for $\delta_0$ given in \eqref{eqn:d0}.
There exists $\{ h_i \}_{i \in \natural}$ such that $h_i \rightarrow 0$ as $i \rightarrow \infty$ and
\begin{align}  \lim \limits_{i \rightarrow \infty} \sup_{t \in [t_1,t_2]}d_H(E_t(h_i,\delta), \odt) = 0 \end{align} for any $0 \leq t_1 < t_2$.
\end{prop}

\medskip

Next, we show the H\"{o}lder continuity in time in Proposition~\ref{lem:ht}. Let us recall some results that concern sets in $S_{r,R}$:

\begin{lem}
\label{lem:ineq}
\cite[Lemma C.1]{KimKwo18}
	For $E_1, E_2 \in S_{r,R}$ and $R>r>0$, the following holds for some $\cK_1=\cK_1(r,R)>0$:
\begin{align}
\label{eqn:ineq}
d_H(E_1, E_2)^{n+1} \leq \cK_1   \td^2(E_1, E_2), \quad d_H(E_1, E_2)^{n+1} \leq \cK_1   \td^2(E_1, E_2).
\end{align}
\end{lem}

\begin{lem}\cite[Lemma 5.3]{KimKwo18}.
\label{lem:gf}	
For $(E_t)_{t\geq0}$ in Definition \ref{def:mmov}, the following holds for some $\cK_2=\cK_2(r_0,R_0)$ and all $0<t_1<t_2$:
\begin{align}
\td^2(E_{t_2}, E_{t_1}) \leq \cK_2 (t_2-t_1)(\cJ_\delta(E_{t_1}) - \cJ_\delta(E_{t_2})).
\end{align}
\end{lem}

Lemma~\ref{lem:ineq} and Lemma~\ref{lem:gf} imply uniform H\"{o}lder continuity in time with respect to $\delta$.
\begin{prop}
\label{lem:ht}
There exists $\cK_3 = \cK_3(r_0,R_0)$, which is independent on $\delta>0$ such that for all $0< t_1 < t_2$, it holds that
\begin{align}
\label{eqn:1ht}
d_H(\odta, \odtb) \leq \cK_3 (t_2 - t_1)^{\frac{1}{n+1}} \per(\oz)^{\frac{1}{n+1}}.
\end{align}
\end{prop}

\begin{proof}
Note that $\cJ_\delta(E_t)$ is nonnegative and decreases in time from the construction of $E_t$ in Definition~\ref{def:mmov}. Thus, Lemma~\ref{lem:gf} implies that
\begin{align}
\label{eqn:ht1}
\td^2(E_{t_2}, E_{t_1}) \leq \cK_2 (t_2-t_1)(\cJ_\delta(E_{t_1}) - \cJ_\delta(E_{t_2})) \leq \cK_2 (t_2-t_1)\cJ_\delta(\Om_{0})
\end{align}
for all $0< t_1 < t_2$ and $\cK_2=\cK_2(r_0,R_0)$ given in Lemma~\ref{lem:gf}.
Note that $|\oz|=1$ implies
\begin{align}
\label{eqn:ht2}
\cJd(\oz) = \per(\oz) + \frac{1}{2\delta} (1 - |\oz|)^2= \per(\oz)
\end{align}
for all $\delta>0$. From Lemma~\ref{lem:ineq} and \eqref{eqn:ht2}, there exists $\cK_3 = \cK_3(r_0,R_0)$ such that for all $0< t_1 < t_2$
\begin{align}
\label{eqn:ht3}
d_H(E_{t_2}, E_{t_1}) \leq \cK_3 (t_2 - t_1)^{\frac{1}{n+1}} \per(\oz)^{\frac{1}{n+1}}.  
\end{align}

\medskip

As $E_t=E_t(h,\delta)$ converges to $\odt$ locally uniformly as $h \rightarrow 0$ and $M \rightarrow \infty$ from Proposition~\ref{thm:geo}, it holds that
\begin{align}
\label{eqn:ht21}
d_H(E_{t_2}, E_{t_1}) \rightarrow d_H(\odtb, \odta) \hbox{ as } h \rightarrow 0, M \rightarrow \infty.
\end{align}
Thus, from \eqref{eqn:ht3} and \eqref{eqn:ht21}, we conclude \eqref{eqn:1ht}
\end{proof}

\noindent {\bf Proof of Theorem~\ref{thm:ap}}
By Proposition~\ref{thm:geo} and Proposition~\ref{lem:ht}, $(\odt)_{t\geq0}$ are equicontinuity on both space and time. Therefore, there exists a sequence $\{\delta_i\}_{i\in\natural}$ such that 
\begin{align}
d_H(\Omega^{\delta_{i}}_t, \oit) \rightarrow 0
\end{align}
locally uniformly in time as $i$ goes to infinity for some $(\oit)_{t\geq0} \subset S_{r_1,R_1}$. By Lemma~\ref{lem-cpt}, we conclude that $|\oit|=1$ for all $t>0$.
\hfill$\Box$

\medskip

Before we finish this section, we show by example that the constraint $S_{r,R}$ on the geometry of $\odt$ is needed to obtain \eqref{eqn:ineq}.

\begin{ex}
\label{ex:ineq}
Consider $\{ E_\e \}_{\e \in (0,1)}$ defined by
\begin{align}
E_\e:= B_1(0) \cup IC(2 e_1,\e)
\end{align}
Here, $IC(x,r)$ is an interior cone defined in \eqref{eqn:ic} and $e_1$ is a unit vector in the positive $x_1$ direction.
Note that $r>0$ satisfying that $\{ E_\e \}_{\e \in (0,1)} \subset S_r$ does not exist.
It holds that $d_H(B_1(0), E_\e) = 1$ but $|B_1(0) \triangle E_\e| \to 0$ as $\e \rightarrow 0$. As $E_r \subset B_2(0)$ for all $\e \in (0,1)$, it holds that
\begin{align}
\td^2(B_1(0), E_\e)) \leq 4 |B_1(0) \triangle E_\e|.
\end{align}
Thus, $\td(B_1(0), E_\e))$ also converges to zero as $\e \rightarrow 0$. Therefore, $\K_1$ satisfying \eqref{eqn:ineq} for $\{ E_\e \}_{\e \in (0,1)}$ does not exist.

On the other hand, Consider $\{ F_{k} \}_{ k \in \natural}$ defined by $F_{k}:=IC((k+1) e_1,1)$, which are not uniformly bounded. By the direct computation, it holds that
\begin{align}
d_H(F_k, B_1(0)) = k \hbox{ and } \td^2(F_k, B_1(0)) \leq nw_n k^2.
\end{align}
where $w_n$ is a volume of an unit ball in $\R^n$.
Thus, we cannot find $\K_1$ such that \eqref{eqn:ineq} holds for $\{ F_{k} \}_{ k \in \natural}$.
\end{ex}

\section{Uniform $L^2$ Estimates of the Lagrange Multiplier and Existence}
\label{sec:l2}

In this section we establish uniform $L^2$ estimates of $\ld$ which yields the weak $L^2$ limit of $\ld$ in Theorem~\ref{thm:l2}. Combining with the stability of viscosity solutions in Theorem~\ref{thm:stf}, we show the existence of solution of \eqref{model} in Corollary~\ref{thm:exi}. 
Following the outline given in \cite{MugSeiSpa16},  the estimates for our constrained discrete gradient flow defined in \eqref{eqn:mmov}. Our new challenge lies in constructing local variations given in Definition~\ref{def:loc} which stays in our admissible set $S_{r_0,R_0}$ (See Lemma~\ref{lem:bdd} and Lemma~\ref{lem:dil}).

\begin{thm}
\label{thm:l2}
Let $\delta \in (0,\delta_0)$ for $\delta_0$ given in \eqref{eqn:d0} and $\ld$ be given in \eqref{modeld}. There exists $\sigma_1=\sigma_1(r_0,R_0)>0$ such that
\begin{align}
\label{eqn:l2}
\| \ld \|_{L^2(\Rpt)}^2 \leq \sigma_1(\per(\oz) + T)
\end{align}
Here, $r_0$ and $R_0$ are given in \eqref{eqn:r0}. As a consequence, there exists a subsequence $\{\delta_{i_j}\}_{j\in \natural}$ of $\{\delta_{i}\}_{i \in \natural}$ in Theorem~\ref{thm:ap} such that $\{ \lambda_{\delta_{i_j}} \}_{j \in \natural}$ weakly converges to $\lam_\infty$ in $L^2([0,T])$ satisfying \eqref{eqn:l2}. 
\end{thm}

Before proving the above theorem, let us show the existence of a viscosity solution of \eqref{model}.

\begin{cor}
\label{thm:exi}
$((\oit)_{t\geq0},\li)$ is a viscosity solution (See Definition~\ref{def:model}) of \eqref{model}. Here, $\oitt$ and $\li$ are given in Theorem~\ref{thm:ap} and Theorem~\ref{thm:l2}, respectively. 
\end{cor}

\begin{proof}
Note that $(\Omega^{\delta}_t)_{t\geq0}$ given in Proposition~\ref{thm:geo} is a viscosity solution of $V = -H + \lam^{\delta}(t)$ in the sense of Definition~\ref{def:visls} (See Remark~\ref{rem:coi}). 
The uniform boundedness of $\{ \ld \}_{j \in \natural}$ in $L^2$ given in Theorem~\ref{thm:l2} implies the equicontinuity of $\{  \Ld \}_{j \in \natural}$ where $\Ld(t):=\int_0^t \ld(s)ds$. 
From Arzela-Ascoli Thereom, 
$\{ \Ld(t) \}_{k \in \natural}$ locally uniformly converges to $\Li(t):=\int_0^t \li(s)ds$ along a subsequence. 
From Theorem~\ref{thm:stf}, Theorem~\ref{thm:ap} and Lemma~\ref{lem:sts}, we conclude that $(\oit)_{t\geq0}$ is a viscosity solution of $V = -H + \li(t)$.

\medskip

On the other hand, from Theorem~\ref{thm:ap}, $|\oit| = |\oz|$ for all $t \geq 0$. Thus, we conclude that $((\oit)_{t\geq0},\li)$ is a viscosity solution of \eqref{model}.
\end{proof}

Let us briefly explain the outline of proof. First, in Proposition~\ref{lem:signed}, we show that $\| d(\cdot,\partial E) \|_{L^2(\partial F)}$ is bounded by $\td(F,E)$ given in \eqref{def:td} up to a constant for any sets $E,F \in S_{r,R}$. The proof is based on the density estimates and Besicovitch's Covering Theorem. 

\medskip

On the other hand, we recall the discrete scheme $E_t=E_t(h,\delta)$ in \eqref{eqn:mmov} and define the corresponding Lagrange multiplier $\ldth$ in \eqref{def:ldth}. In Proposition~\ref{lem:el}, we show that the $\ldth$ is bounded by $\frac{1}{h}\| d(\cdot,\partial E_{t-h}) \|_{L^2(\partial E_t)}$ up to a constant. By combining these two propositions with the inequality from Lemma~\ref{lem:gf}, we conclude that $L^2$ norm of $\ldh$ is uniformly bounded.  Here, we construct a local variation (See Lemma~\ref{lem:bdd} and Lemma~\ref{lem:dil}) in order to find the Euler-Lagrange equation. 

\medskip

Here is density estimates for $S_{r,R}$. We postpone the proof into the Appendix \ref{ap:geo} as the proof is classical.
\begin{lem}
\label{lem:den}
For $E \in S_{r,R}$ and $0<r<R$, the following holds: there exists $\e_0 = \e_0(r,R)$, $\eta_i=\eta_i(r,R)$ for $i=1,2$ and $3$
such that for all $\e \in [0,\e_0]$ and $x \in \partial E$
\begin{align}
\label{eqn:1den}
\eta_1 \e^{n} \leq \min \{ | B_\e(x) \setminus E|,  |E \cap B_\e(x)| \}
\end{align}
and
\begin{align}
\label{eqn:2den}
\eta_3 \e^{n-1} \leq \per(E; B_\e(x)) \leq \eta_2 \e^{n-1}
\end{align}
where $$\per(E;F) := \sup \left\{ \int_E \dv T(x) dx : T \in \C^1_c(F;\Rn), \sup_{F}|T| \leq 1 \right\}$$
\end{lem}

Note that for any $F \subset \R^n$ and $E \subset \R^n$, which has a Lipschitz boundary, it holds that 
\begin{align}
\label{eqn:ph}
P(E;F) = \H^{n-1}(F \cap \partial E)
\end{align}
(See Remark 9.5 and Example 12.6 in \cite{Mag12}).

\medskip

The density estimates in Lemma~\ref{lem:den} and Besicovitch's Covering Theorem in Lemma~\ref{lem:bes} imply the following. A similar inequality was proven for the discrete gradient flow in \cite[Lemma 3.4.1]{MugSeiSpa16}. We extend this results for sets in $S_{r,R}$.

\begin{lem}
\label{lem:bes}
\cite[Theorem 1.27]{EvaGar92}\cite[Theorem 5.1]{Mag12}
(Besicovitch's Covering Theorem)
There exists a positive constant $\xi=\xi(n)$ with the following property: if $\F$ is a family of closed non-degenerate balls of $\R^n$, and the set $\N$ of the centers of the balls in $\F$ is bounded, then there exists at most countable $\F_{1}, \dots, \F_{\xi}$ subfamilies of disjoint balls in $\F$ such that
$$\N \subset \bigcup_{j=1}^{\xi} \bigcup_{B \in \F_j} B.$$
\end{lem}

\begin{prop}
\label{lem:signed}
For $E, F \in S_{r,R}$ and $0<r<R$, the following holds: for some $\sigma_2=\sigma_2(r,R)$
\begin{align}
\label{lem:1signed}
\int_{\partial F} d^2(x,\partial E) d\sigma \leq \sigma_2   \td^2(F, E), 
\end{align}
Here, $\td$ is given in \eqref{def:td}.
\end{prop}

\begin{proof}
{ 1.}
For all $i \in \integer$, define 
\begin{align}
  \D_{i}:= \{ x \in \R^n : 2^{i} < d(x,\partial E) \leq 2^{i+1} \} \hbox{ and } \delta_0:= \min\left\{ \frac{\e_0}{R}, 1 \right\}
\end{align}
where $\e_0$ is given in Lemma~\ref{lem:den}. Let us show that there exists $c_1=c_1(r,R)$ such that for all $x_{i} \in \D_{i} \cap \partial F$
\begin{align}
\label{eqn:signed11}
\I_1 \leq c_1 \I_2, \quad \I_1:=\int_{\partial F \cap B_{2^{i-1} \delta_0}(x_i)} d^2(x,\partial E) d\sigma \hbox{ and } \I_2:=\int_{(E \triangle F) \cap B_{2^{i-1} \delta_0}(x_i)} d(x,\partial E) dx.
\end{align}

\medskip

As $E,F \subset B_R$, it holds that for $2^{i} > 2R$,
\begin{align}
\label{eqn:signed110}
\D_{i} \cap \partial F = \emptyset 
\end{align}
Thus, it is enough to consider $i \leq \log_2 R+1$. Then, it holds that
\begin{align}
\label{eqn:signed111}
2^{i-1} \delta_0 \leq R\delta_0 \leq \e_0
\end{align}

\medskip

For any $x \in B_{2^{i-1} \delta_0}(x_i)$ and $x_{i} \in \D_{i}$, it hold that
\begin{align}
\label{eqn:signed12}
2^{i-1}  \leq d(x,\partial E) \leq 2^{i+2}.
\end{align}
Therefore, $\I_1$ and $\I_2$ are bounded as follows;
\begin{align}
\I_1 \leq \per( F ; B_{2^{i-1} \delta_0}(x_i)) 2^{2i+4} \hbox{ and } \I_2 \geq | (E \triangle F) \cap B_{2^{i-1} \delta_0}(x_i) | 2^{i-1} \delta_0.
\end{align}

\medskip

By \eqref{eqn:signed111} and \eqref{eqn:2den} in Lemma~\ref{lem:den}, it holds that
\begin{align}
\label{eqn:signed13}
\I_1 \leq \eta_2 2^{(i-1)(n-1)} 2^{2i+4} = \eta_2 2^{i(n+1) - n+5 } \delta_0^{n-1}
\end{align}
On the other hand, as $B_{2^{i-1} \delta_0}(x_i) \subset E$ or $B_{2^{i-1} \delta_0}(x_i) \subset E^c$, it holds that
\begin{align}
| (E \triangle F) \cap B_{2^{i-1} \delta_0}(x_i) | =
\begin{cases}
| B_{2^{i-1} \delta_0}(x_i) \setminus F| &\hbox { if } B_{2^{i-1} \delta_0}(x_i) \subset E,\\
| B_{2^{i-1} \delta_0}(x_i) \cap F| &\hbox { if } B_{2^{i-1} \delta_0}(x_i) \subset E^c
\end{cases}
\end{align}
From \eqref{eqn:signed111} and \eqref{eqn:1den} in Lemma~\ref{lem:den}, in both cases, we have
\begin{align}
\label{eqn:signed14}
\I_2 \geq \frac{2^{(i-1)(n+1)}}{\eta_1} = \frac{2^{i(n+1) - n-1}}{\eta_1} \delta_0^{n}.
\end{align}
From \eqref{eqn:signed13} and \eqref{eqn:signed14}, \eqref{eqn:signed11} holds for $c_1:=\frac{2^6 \eta_1 \eta_2 }{\delta_0}$.

\medskip

{ 2.} 
Let $\F:= \{ B_{2^{i-1} \delta_0}(x_i) : x_i \in D_i \}$. Then, by Lemma~\ref{lem:bes}, there exists $\F_{1}, \dots, \F_{\xi}$ subfamilies of disjoint balls in $\F$ such that each family $\F_j$ is at most countable and
\begin{align}
\label{eqn:signed21}
\partial F \cap D_i  \subset \bigcup_{j=1}^{\xi} \bigcup_{B \in \F_j} B.
\end{align}
From \eqref{eqn:signed21} and \eqref{eqn:signed11} in Step 1, it holds that
\begin{align}
\label{eqn:signed22}
\cI_3:=\int_{\partial F \cap D_i} d^2(x,\partial E) d\sigma \leq \sum_{j=1}^{\xi} \sum_{B \in \F_j} \int_{\partial F \cap B} d^2(x,\partial E) d\sigma
\leq  c_1 \sum_{j=1}^{\xi} \sum_{B \in \F_j}  \int_{(E \triangle F) \cap B} d(x,\partial E) dx
\end{align}
As \eqref{eqn:signed12} implies $B \subset  \D_{i-1}\cup \D_{i} \cup \D_{i+1}$ for all $B \in \F$ and $\F_j$ is a family of disjoint balls, we conclude that
\begin{align}
\label{eqn:signed23}
\cI_3 \leq c_1 \sum_{j=1}^{\xi}   \int_{(E \triangle F) \cap (\D_{i-1}\cup \D_{i} \cup \D_{i+1} )} d(x,\partial E) dx  = c_1 \xi   \int_{(E \triangle F) \cap (\D_{i-1}\cup \D_{i} \cup \D_{i+1} )} d(x,\partial E) dx 
\end{align}

\medskip

{ 3.}
From $\bigcup_{i \in \integer} \D_i = \R^n$, \eqref{eqn:signed110} and \eqref{eqn:signed23}, it holds that 
\begin{align*}
\int_{\partial F} d^2(x,\partial E) d\sigma &= \sum_{i \in \integer} \int_{\partial F \cap \D_i} d^2(x,\partial E) d\sigma \leq c_1 \xi \sum_{i \in \integer}    \int_{(E \triangle F) \cap (\D_{i-1}\cup \D_{i} \cup \D_{i+1} )} d(x,\partial E) dx = 3c_1 \xi  \td^2(F, E).
\end{align*}
Thus, \eqref{lem:1signed} holds for $\sigma_2:=3c_1 \xi$.
\end{proof}

\medskip

Now, let us find the Euler-Lagrange equation as \cite[Lemma 3.4.2]{MugSeiSpa16} and \cite[Theorem 17.20]{Mag12}. Consider the discrete flow $E_t=E_t(h,\delta)$ given in \eqref{eqn:mmov} and define the Lagrange multiplier at each time step.
\begin{align}
\label{def:ldth}
\ldth:=\gd(|E_t(h,\delta)|)
\end{align}

\begin{DEF}
\label{def:loc}
\cite[Chapter 17.3]{Mag12} We say that $\{ f_s \}_{-\e_1 < s<\e_2}$ is \textit{a local variation in $A$} for an open set $A$ if for a fixed $-\e_1 < s<\e_2$ and $\e_1,\e_2>0$, $f_s: \R^n \rightarrow \R^n$ is a diffeomorphism of $\R^n$ such that
\begin{align}
f_0(x) = x &\hbox{ for all } x \in \R^n,\\
\{ x \in \R^n : f_s(x) \neq s \} \subset \subset A &\hbox{ for all } -\e_1 < s<\e_2.
\end{align}
Let us denote \textit{the initial velocity of $\{ f_s \}_{-\e_1 < s<\e_2}$} by
\begin{align}
\label{eqn:vel}
\Psi(x):= \frac{\partial f_s }{\partial s}\Big|_{s=0} (x)
\end{align}
 \end{DEF}

Recall the first variation of perimeter and volume from Theorem 17.5 and Proposition 17.8 in \cite{Mag12}. For $E \in S_{r,R}$, it holds that
\begin{align}
\label{eqn:var1}
\per(f_s(E)) &= \per(E) + s \int_{\partial E} \dv_\pe \Psi d\H^{n-1} + O(s^2), \\
\label{eqn:var2}
|f_s(E)| &= |E| + s \int_{\partial E}  \Psi \cdot \n d\H^{n-1} + O(s^2).
\end{align}
where $\dv_\pe$ is the boundary divergence on $\pe$ defined by
\begin{align}
\label{eqn:var3}
\dv_\pe \Psi(x) := \dv \Psi(x) - (\n \cdot \grad \Psi \n) (x)
\end{align}
for $x \in \pe$.
On the other hand, the first variation of $\td$ is as follows,
\begin{align}
\td^2(f_s(E),F) = \td^2(E,F) + s \int_{\partial E}  \sd(x,\partial F) \Psi \cdot \n d\H^{n-1} + O(s^2).
\end{align}
from (3.1) in \cite{MugSeiSpa16}.

\medskip

In our case, the constraints $S_{r_0,R_0}$ gives some difficulties when we choose the local variation. The following two lemmas construct the local variations within the constraint. The first lemma discusses creating a larger perturbed set by dilation. 
For $a>0$, let us denote $a E := \{x: a^{-1}x \in E\}$.

\begin{lem}
\label{lem:bdd}
Let $E_t(h_i,\delta)$ be given in \eqref{eqn:mmov},  $\delta_0$ in \eqref{eqn:d0}, and $r_0,R_0$ in \eqref{eqn:r0}. Then for $0<\delta<\delta_0$ There exists $i^*=i^*(\delta)$ and a constant $s_1>0$  such that for all $i \geq i^*$ and $s \in [0,s_1)$ we have
\begin{align}
\label{eqn:bound21}
(1+s)E_t(h_i,\delta) \in S_{r_0,R_0} \hbox{ for } t\in [0,T].
\end{align}
\end{lem}

\begin{proof}
From Proposition~\ref{thm:geo},  $B_{r_1}(0) \subset \odt \subset B_{R_1}(0)$ for all $t>0$. Let us first show that there exists $i^*=i^*(\delta)$ such that for all $i \geq i^*$, $t \in [0,T]$ we have
\begin{align}
\label{eqn:bdd}
B_{r_2}(0) \subset E_t(h_i,\delta)  \subset B_{R_2}(0), \quad \hbox{ where } R_2:=\frac{R_0 + R_1}{2} \hbox{ and } r_2:=\frac{r_0 + r_1}{2}
\end{align}
By the uniform convergence of $E_t(h_i)$ in $[0,T]$ from Proposition~\ref{prop:nom}, there exists $i^*=i^*(\delta)$ such that
\begin{align}
d_H(E_t(h_i),\odt) \leq \min \left\{\frac{R_0 - R_1}{4} , \frac{r_0 - r_1}{4} \right\}
\end{align}
for all $i \geq i^*$ and $t \in [0,T]$. This implies \eqref{eqn:bdd}. From \eqref{eqn:bdd}, we conclude that for all $s \in [0,s_1)$
\begin{align}
(1+s)E_t(h_i,\delta)\in B_{R_0} \hbox{ where } s_1:=\frac{R_0}{R_2}-1
\end{align}

\medskip

 As $E_t \in S_{r_0}$, \eqref{eqn:ic} imply that for all $ x \in \pet$
\begin{align}
IC(r,x) \subset E_t.
\end{align}
Since $(1+s)IC(r,x) = IC((1+s)r,(1+s)x))$, we conclude that for all $x \in \partial (1+s)(E_t)$
\begin{align}
IC((1+s)r, x) \subset (1+s)(E_t)
\end{align}
As $IC(r,x) \subset IC((1+s)r, x)$, \eqref{eqn:bound21} holds for $s \in [0,s_1)$.
\end{proof}

Generating a smaller set that stays in $S_{r_0,R_0}$ turns out to be more delicate. For this we need perturbations that preserve $\partial B_{r_0}(0)$ and shrinks outside of $B_{r_0}(0)$.  To stay within $S_{r_0,R_0}$ we must ensure that the angles of interior cone and exterior cone given in \eqref{eqn:ic} and \eqref{eqn:ec} do not decrease for the perturbed set. This is what we prove with a specific choice of the perturbation $\G_s$ below.

\begin{lem}
\label{lem:dil}
Let $E_t(h_0,\delta)$, $\delta_0, r_0,R_0$ and $i^*$ be as in the previous lemma. Let us define 
$$\psi:=\chi_{E_t(h_i,\delta)} - \chi_{E_t(h_i,\delta)^C} \hbox{ and } \Gps(x) := \psi((1+s(|x|^2 - r_0^2))x).
$$
Then for $\delta \in (0,\delta_0)$ there exists $s_2>0$ such that 
\begin{align}
\{\Gps>0\} \in S_{r_0,R_0}\hbox{ for } s\in [0, s_2),  i \geq i^ * \hbox{ and } t\in [0,T].
\end{align}

\end{lem}

\begin{proof} We may assume that $E_t$ has a $\C^1$ boundary. Then, there is a $\C^1$ function $\phi : \R^n \rightarrow \R$ such that 
\begin{align}
\label{eqn:dil01}
\{ \phi > 0 \} = E_t, \quad  \{ \phi = 0 \} =  \partial E_t, \quad \{ \phi < 0 \} =  \overline{E}_t^C, \hbox{ and } D\phi \neq 0 \hbox{ on } \pet. 
\end{align}

\medskip

First note that  as $E_t \in S_{r_0,R_0}$
we have
$B_{r_0} \subset \{\Gp>0\} \subset B_{R_0}$.
To show  that $\{\Gp>0\}$ is in $S_r$, from Lemma~\ref{lem:starloc} it is enough to show that 
\begin{align}
\label{eqn:dil10}
D\Gp(x) \neq 0 \hbox{ and }  - \frac{D\Gp}{|D\Gp|}(x) \cdot x \geq r_0  \hbox{ for all } x\in \{\Gp=0\}.
\end{align}

For the rest of the proof we assume that $x\in \{\Gp=0\}$. 

\medskip

Denote $P_s(x):=1+s(|x|^2 - r_0^2)$ so that we can write $\Gp(x) = \phi(P_s(x)x)$, and thus $P_s(x)x\in \{\phi=0\}$ with  $D\phi(P_s(x)x)\neq 0$. Observe that
\begin{align}
\label{eqn:s2}
1 \leq P_s(x) \leq \frac{3}{2}  \hbox{ for } 0\leq s < s_2 := \frac{1}{2(R_0^2 - r_0^2)}.
\end{align}
Since \begin{align}
\label{eqn:dil221}
- D\Gp(x) \cdot x = - (|x|^2s + P_s(x)) D\phi(P_s(x)x) \cdot x,
\end{align}
we have 
\begin{align}
\label{eqn:dil12}
|D\Gp(x)|^2 &= P_s(x)^2 |D\phi(P_s(x) x)|^2 + 4s (|x|^2s + P_s(x)) (D\phi(P_s(x)x) \cdot x)^2.
\end{align}
 \eqref{eqn:s2} yields
\begin{align}
|D\Gp(x)|^2 \geq |D\phi(P_s(x) x)|^2 >0 \hbox{ for all } s \in [0,s_2),
\end{align}
and thus the first condition of \eqref{eqn:dil10} is satisfied.

\medskip

 Let us now show the second condition of \eqref{eqn:dil10}.
As $\{ \phi >0 \} \in S_{r_0,R_0}$ and $P_s(x)x \in \{ \phi = 0 \}$, Lemma~\ref{lem:starloc} implies
\begin{align}
\label{eqn:dil22}
-   \frac{D\phi}{|D\phi|}(P_s(x) x) \cdot (P_s(x)x) \geq r_0.
\end{align}
From \eqref{eqn:dil221}

\eqref{eqn:s2} and \eqref{eqn:dil22} imply that $- D\Gp(x) \cdot x$ is positive.
Thus, it is enough to show that
\begin{align}
\I_1 := (- D\Gp(x) \cdot x )^2 - r_0^2 |D\Gp(x)|^2 \geq 0 \hbox{ for all } s \in [0,s_2) \hbox{ and } x \in \{\Gp=0\}. 
\end{align}

\medskip

From \eqref{eqn:dil22} and \eqref{eqn:dil12},  it holds that

\begin{align}
r_0^2 |D\Gp(x)|^2 \leq \left(P_s(x)^4  + 4 r_0^2s( |x|^2s + P_s(x)) \right) (D\phi(P_s(x)x) \cdot x)^2. 
\label{eqn:dil23}
\end{align}
From  \eqref{eqn:dil221} and \eqref{eqn:dil23} it follows that
\begin{align}
\frac{\I_1}{(D\phi(P_s(x)x) \cdot x)^2} &\geq \left (P_s(x)^2 +  4 |x|^2s( |x|^2 s + P_s(x)) \right) - \left(P_s(x)^4 + 4 r_0^2s( |x|^2 s + P_s(x))\right).
\label{eqn:dil24}
\end{align}
Using $s(|x|^2 - r_0^2) = P_s(x)-1$ and factorizing the above, we conclude 
\begin{align}
\frac{\I_1}{(D\phi(P_s(x)x) \cdot x)^2} \geq (P_s(x)-1) ( - P_s(x)^3 - P_s(x)^2 + 4P_s(x)  + 4|x|^2 s ).
\end{align}
From \eqref{eqn:s2}, we conclude that $\I_1 \geq 0$ for all $s \in [0,s_2)$ and $x \in \{\Gp=0\}$.
\end{proof}

From Lemma~\ref{lem:bdd} and \ref{lem:dil}, we get the following estimates.

\begin{prop}
\label{lem:el}
There exists $\sigma_3 = \sigma_3(r_0,R_0)$ and $\sigma_4 = \sigma_4(r_0,R_0)$ such that for all $t \in \Rpt$ and $i \geq i^*$,
\begin{align}
\label{eqn:el}
| \lambda_\delta^{h_i}(t) |^2 &\leq \sigma_3 +   \frac{\sigma_4}{h^2}  \int_{\pet(h_i,\delta)} d^2(x, \partial E_{t-h_i})  d\sigma  
\end{align}
Here, $i^*$ is given in Lemma~\ref{lem:bdd} and $E_t(h_i,\delta)$ is given in Proposition~\ref{prop:nom}. Also, $r_0$ and $R_0$ are given in \eqref{eqn:r0}.
\end{prop}

\begin{proof}

For simplicity, let $h = h_i$ for $i \geq i^*$.

\medskip

{ 1.} First, show that if $f_s(E_t) \in S_{r_0,R_0}$ for all $s\in [0,s_0)$, then it holds that
\begin{align}
\label{eqn:bound11}
\ldth \int_{\pet} \n \cdot \Psi d\sigma \leq  \int_{\pet} \dv_{\pet} \Psi +  \frac{1}{h} \sd(x, \partial E_{t-h}) \n \cdot \Psi d\sigma.
\end{align}
As $E_t$ is a minimizer of $\cJ_\delta(\cdot) + \frac{1}{h}\tilde{d}^2(\cdot,E_{t-h})$ on $S_{r_0,R_0}$, \eqref{eqn:var1}, \eqref{eqn:var2}, and \eqref{eqn:var3} imply that
\begin{align}
s \ldth \int_{\pet} \n \cdot \Psi d\sigma \leq s \int_{\pet} \dv_{\pet} \Psi d\sigma + \frac{s}{h} \int_{\pet}  \sd(x, \partial E_{t-h}) \n \cdot \Psi d\sigma + O(s^2)
\end{align}
for all $s \in [0,s_0)$. Dividing both sides by $s>0$ and sending $s$ to zero, we conclude \eqref{eqn:bound11}. 

\medskip

{ 2.} Let us find the upper bound of $\ldth$. Recall $f_s(x):=x+sx$ in Lemma~\ref{lem:bdd}. Then, $f_s(E_t) \in S_{r_0,R_0}$ for $s \in [0,s_1)$ and $\Psi(x)=x$.  From \eqref{eqn:bound11} in Step 1 and $r_0 \leq \n \cdot x \leq R_0$ on $\pet$, it holds that
\begin{align}
\label{eqn:bound22}
\ldth &\leq \frac{ \int_{\pet} \dv_{\pet} \Psi +  \frac{1}{h} \sd(x, \partial E_{t-h}) \n \cdot \Psi d\sigma}{\int_{\pet} \n \cdot \Psi d\sigma} \leq \frac{n-1}{r_0} + \frac{R_0}{r_0 \per(E_t)} \frac{1}{h} \int_{\pet} \sd(x, \partial E_{t-h})  d\sigma
\end{align}

\medskip

{ 3.} Let us construct the lower bound. Define $g, f : \R^n \times [0,s_2) \rightarrow \R^n$ by 
\begin{align}
g_s(x) = g(x,s) := (1+s(|x|^2 - r^2))x \hbox{ and } f_s(x) = f(x,s) := (g_s)^{-1}(x)
\end{align}
where $s_2$ is given in \eqref{eqn:s2} in Lemma~\ref{lem:dil}. As $g(f(x,s),s)=x$ and $Dg_s|_{s=0} = I$, it holds that
\begin{align}
\frac{\partial f_s }{\partial s}\Big|_{s=0} (x) = -\frac{\partial g_s }{\partial s}\Big|_{s=0} (x) = -(|x|^2 - r^2)x
\end{align}
From the above and \eqref{eqn:vel}, the initial velocity is
\begin{align}
\Psi(x) = -(|x|^2 - r^2)x.
\end{align}

\medskip

From Lemma~\ref{lem:dil}, $f_s(E_t) \in S_{r_0,R_0}$ for $s \in [0,s_2)$. By \eqref{eqn:bound11} and $ \Psi \cdot \n \leq 0$ on $\pet$, it holds that
\begin{align}
\label{eqn:bound31}
\ldth &\geq \frac{ \int_{\pet} \dv_{\pet} \Psi +  \frac{1}{h} \sd(x, \partial E_{t-h}) \n \cdot \Psi d\sigma}{\int_{\pet} \n \cdot \Psi d\sigma} 
\end{align}
Note that from \eqref{eqn:bdd}
\begin{align}
\label{eqn:bound32}
-R_0(R_0^2-r_0^2) \leq \Psi \cdot \n \leq -r_0(r_2^2-r_0^2) \hbox{ and } -(n+1)(R_0^2 - r_0^2) \leq \dv_\pet  \Psi \leq -(n-1) (r_2^2 - r_0^2).
\end{align}
From \eqref{eqn:bound31} and \eqref{eqn:bound32}, we conclude that
\begin{align}
\label{eqn:bound33}
\ldth &\geq  \frac{(n-1) (r_2^2 - r_0^2)}{R_0(R_0^2-r_0^2)} -  \frac{1}{h \per(E_t)} \int_{\pet} d(x, \partial E_{t-h})  d\sigma 
\end{align}

\medskip

{ 4.} From \eqref{eqn:bound22} and \eqref{eqn:bound33}, there exists $c_1=c_1(r_0,R_0)$ and $c_2=c_2(r_0,R_0)$ such that
\begin{align}
\label{eqn:bound41}
|\ldth|  \leq c_1 +  \frac{c_2}{h\per(E_t)} \int_{\pet} d(x, \partial E_{t-h})  d\sigma.
\end{align}

\medskip

From \eqref{eqn:bound41}, $(a+b)^2 \leq 2(a^2 + b^2)$ for $a,b \in \R$ and the H\"{o}lder's inequality, it holds that
\begin{align}
| \ldth |^2  &\leq 2 c_1^2 +  \frac{2c_2^2}{h^2\per(E_t)^2}  \left( \int_{\pet} d(x, \partial E_{t-h})  d\sigma \right)^2 \leq 2c_1^2 + \frac{2c_2^2}{h^2 \per(E_t)}   \int_{\pet} d(x, \partial E_{t-h})^2  d\sigma 
\end{align}
By the isoperimetric inequality and $B_{r_0} \subset E_t$, we have $\per(E_t) > c_3$ for some $c_3 = c_3(r_0)$,
we conclude that \eqref{eqn:el} holds for
\begin{align}
\sigma_3: = 2c_1^2 \hbox{ and } \sigma_4: = \frac{2c_2^2}{c_3}.
\end{align}
\end{proof}

\noindent {\bf Proof of Theorem~\ref{thm:l2}}

\medskip

Let us show that $\| \ldh \|_{L^2(\Rpt)}^2$ is uniformly bounded for all $h = h_i$ for $i \geq i^*$ and all $\delta \in (0,\delta_0)$. Here, $\{ h_i \}_{i \in \natural}$ is given in Proposition~\ref{prop:nom} and $i^* = i^*(\delta)$ is given in Lemma~\ref{lem:bdd}.

\medskip

By Proposition~\ref{lem:el} and Proposition~\ref{lem:signed}, it holds that
\begin{align}
\label{eqn:l211}
\| \ldh \|_{L^2(\Rpt)}^2 &\leq \sigma_3 T +   \frac{\sigma_2\sigma_4}{h^2} \int_0^T \td^2(E_t,E_{t-h}) dt \leq \sigma_3 T + \frac{\sigma_2\sigma_4}{h} \sum_{k=1}^{\left[ \frac{T}{h}\right]} \td^2(E_{kh},E_{(k-1)h}).
\end{align}

\medskip

Note that Lemma~\ref{lem:gf} implies
\begin{align}
\label{eqn:l212}
\frac{1}{h}\sum_{k=1}^{\left[ \frac{T}{h}\right]} \td^2(E_{kh},E_{(k-1)h}) &\leq \K_2  \sum_{k=1}^{\left[ \frac{T}{h}\right]} (\Jd(E_{(k-1)h}) - \Jd(E_{kh})) = \K_2  (\J(\oz) - \J(E_{\left[ \frac{T}{h}\right] h})) \leq \K_2  \per(\oz)
\end{align}
Thus, \eqref{eqn:l211} and \eqref{eqn:l212} imply that
\begin{align}
\| \ldh \|_{L^2(\Rpt)}^2 \leq \sigma_3 T +\sigma_2\sigma_4 \K_2 \per(\oz)
\end{align}
for all $h = h_i$ for $i \geq i^*$.

\medskip

By the uniform continuity of $E_t(h,\delta)$ in Proposition~\ref{prop:nom}, $\ldh$ given in \eqref{def:ldth} uniformly converges to $\ld$ given in \eqref{modeld} along a subsequence. Thus, we conclude that \eqref{eqn:l2} holds for
\begin{align}
\label{eqn:sig1}
\sigma_1:=\max\{ \sigma_3, \sigma_2\sigma_4 \K_2 \}. 
\end{align}
Here, $\sigma_2$ is given in Proposition~\ref{lem:el}, $\sigma_3$ and $\sigma_4$ are given in Proposition~\ref{lem:signed} and $\K_2$ is given in Lemma~\ref{lem:gf}.
For $\delta_{i} \in (0,\delta_0)$ given in Theorem~\ref{thm:ap}, $\lambda_{\delta_{i}}$ is uniformly bounded for all $i \in \natural$. Thus, by Banach-Alaoglu Theorem, there exists a subsequence $\delta_{i_j}$ of $\delta_{i}$ in Theorem~\ref{thm:ap} such that $\lambda_{\delta_{i_j}}$ weakly converges to $\lam_\infty$ in $L^2[0,T]$.
\hfill$\Box$

\medskip

For the later purpose in Section~\ref{sec:cov}, let us also construct $L^2$ estimates in $[t_0,t_0+T]$ for all $t_0 \geq 0$.

\begin{cor}
\label{cor:l2u}
Let $\delta \in (0,\delta_0)$ for $\delta_0$ given in \eqref{eqn:d0} and $\ld$ be given in \eqref{modeld}. \begin{align}
\label{eqn:1l2u}
\| \ld \|_{L^2([t_0,t_0+T])}^2 \leq \sigma_1(\per(\oz) + T)
\end{align}
where $\sigma_1$ is given in \eqref{eqn:sig1}.
\end{cor}

\begin{proof}
As $\Jd(\odt)$ given in \eqref{eqn:Jd} decreases in time, $\Jd(\odt)$ is bounded by $\Jd(\oz)=\per(\oz)$ for all $\delta >0$ and $t\geq 0$. 
From \eqref{eqn:l211} and \eqref{eqn:l212} in the proof of Theorem~\ref{thm:l2}, we have
\begin{align}
\| \ldh \|_{L^2([t_0,t_0+T])}^2 \leq \sigma_1(\per(\oz) + T)
\end{align}
where $\sigma_1$ is given in \eqref{eqn:sig1}. As the proof of Theorem~\ref{thm:l2}, we conclude \eqref{eqn:1l2u} 
\end{proof}

\medskip

\section{Large-Time Behavior}
\label{sec:cov}

In this section, we discuss the large-time behavior of $\oitt$ given in Theorem~\ref{thm:ap}. Here is the main theorem in this section.

\begin{thm}
\label{thm:cov}
$\oitt$ given in Theorem~\ref{thm:ap} uniformly converges to a ball of volume $1$, modulo translation. More precisely 
\begin{align}
\label{prop:1lim} 
\inf\left\{ d_H(\oit, B_{r_\infty}(x)): x\in \oB_{r_1}(0)\right\} \to 0 \hbox{ as } t\to\infty,
\end{align}
where $r_1$ is given in Proposition~\ref{thm:geo}, $r_\infty := (w_n)^{-\frac{1}{n}}$ and $w_n$ is a volume of an unit ball in $\R^n$.
\end{thm}

Intuitively this convergence is due to the flow's formal gradient flow structure with respect to the perimeter energy. Unforunately, due to the lack of uniform regularity for $\odt$ with respect to $\delta>0$, we are not able to directly show that $\oit$ is the gradient flow of the perimeter energy in the space of sets with unit volume.   Hence we instead utilize the gradient flow structure for the $\delta$-flow, as given in section 4, to show this convergence. 
\medskip

The main estimate in the analysis is Lemma~\ref{lem:per}, where we bound the difference of total perimeter with respect to their differences in Hausdorff distance, in the class of star-shaped sets with their total curvature in $L^2$. Based on this estimate, we can proceed to show in \eqref{eqn:lim21} that the time integral of $\delta$-energy converges to the time integral of the perimeter energy. This now establishes the link between the gradient flow structure of $\delta$-flow and the limit flow, and the asymptotic convergence follows.

\medskip

For $k\in \mathbb{N}$ we consider $((U^k_t)_{t \geq 0}, \eta^k)$ defined by 
\begin{align}
\label{eqn:uet}
U^k_t := \Omega^{\infty}_{t+k} \hbox{ and } \eta^k(t) := \li(t+k). 
\end{align}
Here, $\oitt$ and $\li$ are given in Theorem~\ref{thm:ap} and Theorem~\ref{thm:l2}, respectively.

\begin{prop}
\label{prop:usu}
There exists a subsequence $\{ k_i \}_{i \in \natural}$ such that $\{ (U^{k_i}_t)_{t \geq 0} \}_{ i \in \natural}$ locally uniformly converges to $(U^\infty_t)_{t \geq 0} \subset S_{r_1,R_1}$ and $\{ \eta^{k_{i}} \}_{i \in \natural}$ weakly converges to $\eta^\infty$ in $L^2(\Rpt)$ for all $T>0$. As a consequence, $(U^\infty_t)_{t \geq 0}$ is a viscosity solution (See Definition~\ref{def:visls}) of $V = -H + \eta^\infty(t)$. Here, $r_1$ and $R_1$ are given in Proposition~\ref{thm:geo}.
\end{prop}

\begin{proof} 
Theorem~\ref{thm:ap} and Proposition~\ref{lem:ht} imply that for all $0< k_1 < k_2$
\begin{align}
\label{eqn:lim22}
d_H(U^{k_1}_t, U^{k_2}_t) \leq \cK_3 (k_2 - k_1)^{\frac{1}{n+1}} \per(\oz)^{\frac{1}{n+1}}.
\end{align}
where $\cK_3$ is given in Proposition~\ref{lem:ht}. Also, as $ \{ U^k_t \}_{k \in \natural} \subset S_{r_1,R_1}$ from Theorem~\ref{thm:ap}, we have the equicontinuity in both space and time of $\{ U^k_t \}_{k \in \natural}$.

\medskip

By the equicontinuity of $\{ U^k_t \}_{k \in \natural}$ and the uniform $L^2$ estimates in Corollary~\ref{cor:l2u}, there exists a subsequence $\{ k_i \}_{i \in \natural}$ such that $\{ (U^{k_i}_t)_{t \geq 0} \}_{ i \in \natural}$ locally uniformly converges to $(U^\infty_t)_{t \geq 0} \subset S_{r_1,R_1} $ and $\{ \eta^{k_{i}} \}_{i \in \natural}$ weakly converges to $\eta^\infty$ in $L^2(\Rpt)$ for all $T>0$.

\medskip

Note that $(U^k_t)_{t \geq 0}$ is a viscosity solution of $V=-H + \eta^k(t)$. From Theorem~\ref{thm:stf} and Lemma~\ref{lem:sts}, $(U^\infty_t)_{t \geq 0}$ is a viscosity solution of $V = -H + \eta^\infty(t)$. 
\end{proof}

Now, in Lemma~\ref{lem:per}, we estimates the time integral of the perimeter difference for two evolving sets $\ottj \subset S_{r,R}$ and $j \in \{1,2\}$.

\begin{lem}
\label{lem:per}
For $j \in \{1,2\}$, consider $\ottj \subset S_{r,R}$ for $R>r>0$ such that $(\potj)_{t>0}$ are smooth.
Suppose that there exists a constant $\W < +\infty$ such that for $T>0$ and $j \in \{1,2\}$
\begin{align}
\label{eqn:w}
\int_{0}^T \int_{\potj} H(x,t)^2 d\sigma dt < \W
\end{align}
where $H(x,t)$ is the mean curvature at $x \in \potj$. Then, there exists a constant $\m=\m(r,R,T,\W)>0$
such that
\begin{align}
\label{eqn:1per}
\left( \int_{0}^T \per(\ota) - \per(\otb) dt \right)^2 \leq \m \sup_{t \in [0,T]} d_H(\ota,\otb).
\end{align}
\end{lem}

\begin{proof} As $\{ \otj \}_{j \in \{1,2\}} \subset S_{r,R}$ are smooth for $t>0$, there exist two smooth functions $u_1, u_2 : B_r^{n-1}(0) \times [0,T] \to \R$ such that 
for $j=1,2$
\begin{align}
\label{eqn:per10}
\potj \cap C_{r,R}^+(0) = \{ (u_j(y',t),y') : y' \in B_r^{n-1}(0) \}
\end{align}
Furthermore, from $\Omega_1, \Omega_2 \in S_{r,R}$ again, there exists a constant $c_1 = c_1(r,R)$ 
such that
\begin{align}
\label{eqn:per11}
\| u_1 - u_2 \|_{L^\infty(B_r^{n-1}(0) \times [0,T])} \leq c_1 \sup_{t \in [0,T]}d_H(\ota, \otb) \hbox{ and } \| \nabla u_j \|_{L^\infty(B_r^{n-1}(0) \times [0,T])} \leq c_1 \hbox{ for } j=1,2.
\end{align}

\medskip

{ 1.} Let us first show that there exists $\m_1 = \m_1(r,R, T, \W)$ for $\W$ given in \eqref{eqn:w}
\begin{align}
\label{eqn:per111}
\| \n_1 - \n_2 \|_{L^2(\Brt)}^2 
\leq \m_1 \sup_{t \in [0,T]}d_H(\ota, \otb)  \hbox{ where }  \n_j:= \frac{(1,\nabla u_j)}{\sqrt{ 1+ |\nabla u_j|^2 }} \hbox{ for } j=1,2.
\end{align}
As $\n_1$ and $\n_2$ are unit vectors, we get the following by the direct computation,
\begin{align}
| \n_1 - \n_2 |^2 &= 2( 1 - \n_1 \cdot \n_2),\\ &\leq (\sqrt{ 1+ | \nabla u_1 |^2}  + \sqrt{ 1+ | \nabla u_2 |^2} ) ( 1 - \n_1 \cdot \n_2),\\
&= ((\sqrt{ 1+ | \nabla u_1 |^2}) \n_1  - (\sqrt{ 1+ | \nabla u_2 |^2}) \n_2) \cdot (\n_1 - \n_2) = \nabla(u_1 - u_2) \cdot (\n_1' - \n_2'). \label{eqn:per12}
\end{align}
where $\n_j'$ is the last $n-1$ components of $\n_j$ given by
\begin{align}
\label{eqn:per121}
\n_j' := \frac{\nabla u_j}{\sqrt{ 1+ |\nabla u_j(x)|^2 }} \hbox{ for } j \in \{1,2\}.
\end{align}

\medskip

Note that the mean curvature at $(u_j(x,t),x) \in \potj$ for $x \in B_r^{n-1}(0)$ is given by
\begin{align}
H((u_j(x,t),x),t) = \nabla \cdot \n_j'(x,t).
\end{align}
From \eqref{eqn:w}, there exists $c_2 = c_2(r,R,\W)$ such that for $j\in \{1,2\}$
\begin{align}
\label{eqn:per13}
\| \nabla \cdot \n_j' \|_{L^2(\Brt)} \leq c_2
\end{align}
From integration by parts, we have
\begin{align}
\I_1 &:= \int_{\Brt} \nabla(u_1 - u_2) \cdot (\n_1' - \n_2') dx dt,\\
&= \int_{\partial \Brt} (u_1 - u_2) (\n_1' - \n_2') \cdot \nu d\sigma dt - \int_{\Brt} (u_1 - u_2) \cdot (\nabla \cdot (\n_1' - \n_2')) dx dt \label{eqn:per14}
\end{align}
where $\nu$ is the outward normal vector on $\partial B_r^{n-1}(0)$. By applying the H\"{o}lder inequality at each terms and using \eqref{eqn:per11} and \eqref{eqn:per13}, we have
\begin{align}
|\I_1| &\leq   (2 \per(B^{n-1}_r)T + \| \nabla \cdot (\n_1' - \n_2') \|_{L^2(\Brt)} |B^{n-1}_r|^{\frac{1}{2}} T^{\frac{1}{2}}  )  \| u_1 - u_2 \|_{L^\infty(\Brt)},\\
&\leq  (2 \per(B^{n-1}_r)T + 2c_2 |B^{n-1}_r|^{\frac{1}{2}} T^{\frac{1}{2}} ) c_1 \sup_{t \in [0,T]}d_H(\ota, \otb). \label{eqn:per141}
\end{align}

\medskip

From \eqref{eqn:per12} and \eqref{eqn:per141}, we conclude \eqref{eqn:per111} with
\begin{align}
\label{eqn:per15}
\m_1:= 2c_1 \left(\per(B^{n-1}_r)T + c_2 |B^{n-1}_r|^{\frac{1}{2}} T^{\frac{1}{2}} \right)
\end{align}

\medskip

{ 2.} Let us show that there exists $\m_2 = \m_2(r,R, T, \W)$ for $\W$ given in \eqref{eqn:w}
\begin{align}
\label{eqn:per24}
 \left( \I_2 \right)^2 \leq \m_2 \sup_{t \in [0,T]} d_H(\ota,\otb) \hbox{ where } \I_2 :=  \int_0^T \per(\Omega_1 ; C_{r,R}^+(0)) - \per(\Omega_2 ; C_{r,R}^+(0)) dt 
\end{align}
Recall from \eqref{eqn:ph} and Theorem 9.1 in \cite{Mag12}, we have
\begin{align}
\int_0^T \per(\otj ; C_{r,R}^+(0)) dt = \int_{\Brt} \sqrt{ 1+ |\nabla u_j |^2 } dxdt = \int_{\Brt} (1,\nabla u_j) \cdot \n_j dxdt
\end{align}
where $\{ u_j \}_{j\in \{1,2\}}$ and $\{ n_j \}_{j\in \{1,2\}}$ are given in \eqref{eqn:per10} and \eqref{eqn:per111}, respectively. 
By adding and subtracting the same term in $\I_2$, we have the identity
\begin{align}
\label{eqn:per21}
\I_2 = \I_3 + \I_4 
\end{align}
where
\begin{align}
\I_3:= \int_{\Brt} (1,\nabla u_1) \cdot ( \n_1 - \n_2 ) dxdt 
\end{align}
and
\begin{align}
\I_4:= \int_{\Brt} \left( (1,\nabla u_1) - (1,\nabla u_2) \right) \cdot \n_2  dx dt  = \int_{\Brt} \nabla (u_1 - u_2) \cdot \n_2' dx dt 
\end{align}
Here, $\{ n_j \}_{j\in \{1,2\}}$ and $\{ n_j' \}_{j\in \{1,2\}}$ are given in \eqref{eqn:per111} and \eqref{eqn:per121}, respectively.

\medskip

By applying \eqref{eqn:per11} and \eqref{eqn:per111} and the H\"{o}lder inequality, we get 
\begin{align}
\label{eqn:per22}
\I_3^2 \leq (1+c_1^2) |B^{n-1}_r| T  \m_1 \sup_{t \in [0,T]}d_H(\ota, \otb).
\end{align}
where $c_1$ and $\m_1$ are given in \eqref{eqn:per11} and \eqref{eqn:per15}.
On the other hand, by the similar arguments in \eqref{eqn:per14} 
\begin{align}
\label{eqn:per221}
\I_4 \leq \m_1 \sup_{t \in [0,T]} d_H(\ota, \otb)
\end{align}
where $\m_1$ is given in \eqref{eqn:per15}.
As $\ottj \subset S_{r,R}$ for $j \in \{1,2\}$, we have
\begin{align}
\label{eqn:per222}
\sup_{t \in [0,T]} d_H(\ota, \otb) \leq 2R.
\end{align}
Thus, \eqref{eqn:per221} and \eqref{eqn:per222} imply that
\begin{align}
\label{eqn:per23}
\I_4^2 \leq 2 \m_1^2 R d_H(\ota, \otb).
\end{align}

\medskip

From \eqref{eqn:per21} combining with \eqref{eqn:per22} and \eqref{eqn:per23}, we have
\begin{align}
\I_2^2 \leq 2(\I_3^2 + \I_4^2) \leq 2((1+c_1^2)  |B^{n-1}_r| T  \m_1 + 2 \m_1^2 R)  \sup_{t \in [0,T]}d_H(\ota, \otb)
\end{align}
Thus, we conclude \eqref{eqn:per24}
for 
\begin{align}
\m_2 := 2\m_1 \left( (1+c_1^2)  |B^{n-1}_r| T  \m_1 + 2 \m_1^2 R \right) .
\end{align}
Here, $c_1$ and $\m_1$ are given in \eqref{eqn:per11} and \eqref{eqn:per15}.

\medskip

{ 3.}
As every sets in $S_{r,R}$ can be covered by a finite number of cylinders $C_{r,R}^+(0)$ after some rotations, \eqref{eqn:per24} implies \eqref{eqn:1per}.
\end{proof}

From the estimates in Lemma~\ref{lem:per} and our approximation from $\odtt$ in Theorem~\ref{thm:ap}, we conclude that the limit flow $(U^\infty_t)_{t \geq 0}$ is stationary.

\begin{prop}
\label{prop:lon}
$(U^\infty_t)_{t \geq 0}$ given in Proposition~\ref{prop:usu} is stationary.
\end{prop}

\begin{proof}
{ 1.}
Let us show that there exists $\E^\infty: \R^+ \to \R^+$ such that
\begin{align}
\label{eqn:lim21}
\E^\infty(k):= \lim_{\delta \to 0} \E^\delta(k) \hbox{ where } \E^\delta(k):= \int_0^T \cJ_\delta(\Omega_{t+k}^{\delta}) dt.
\end{align}
It is enough to show that $\{ \E^\delta(k) \}_{\delta > 0}$ is a Cauchy sequence as $\delta \to 0$ for all $k \in \Rpz$. As $\odt$ is smooth for $t>0$ from Proposition~\ref{thm:geo} and $\odt$ is a gradient flow of $\Jd$, we have
\begin{align}
\label{eqn:lim23}
\int_{t_0}^{t_0+T} \int_{\podt} V^2 d\sigma dt = \Jd(\Omega_{t_0}^{\delta}) - \Jd(\Omega_{t_0+T}^{\delta})  \leq \per(\oz).
\end{align}
where $V$ is the normal velocity at $x \in \podt$. As $H = \lam - V$, Corollary~\ref{cor:l2u} and \eqref{eqn:lim23} implies the uniform bound on $\| H \|_{L^2(t_0,t_0+T; L^2(\podt))}$.

\medskip

As $\cJ_\delta(\odt) = \per(\odt) + 2 \delta \ldt^2$, Lemma~\ref{lem:per} and  Corollary~\ref{cor:l2u} imply that for $\delta_1 > \delta_2 > 0$
\begin{align}
\left| \E^{\delta_1}(k) - \E^{\delta_2}(k) \right| &\leq  \left| \int_0^T \per(\Omega_{t+k}^{\delta_1}) - \per(\Omega_{t+k}^{\delta_2}) dt \right| + 2 \delta_1 \| \lambda_{\delta_1} \|_{L^2([k,k+T])}^2 + 2 \delta_2 \| \lambda_{\delta_2} \|_{L^2([k,k+T])}^2,\\ &\leq  c \left(\sup_{t \in [0,T]} d_H(\Omega^{\delta_1}_t, \Omega^{\delta_2}_t)^{\frac{1}{2}} + \delta_1 + \delta_2 \right)
\end{align}
where a constant $c$ is given by
\begin{align}
c := \max\left\{ \m^{\frac{1}{2}}, 2 \sigma_1 (\per(\oz) + T) \right\}.
\end{align}
Here, $\sigma_1$ and $\m$ are given in \eqref{eqn:sig1} and \eqref{eqn:1per}, respectively. From Theorem~\ref{thm:ap}, we conclude \eqref{eqn:lim21}.
 
\medskip

{ 2.} Lemma~\ref{lem:gf} and the smoothness of $\odt$ for $t>0$ from Proposition~\ref{thm:geo} imply that for $s,k \in \Rpz$
\begin{align}
\int_0^T \td^2(\Omega_{t+k+s}^{\delta}, \Omega_{t+k}^{\delta}) dt \leq s \cK_2  (\E^\delta(k) - \E^\delta(k+s))
\end{align}
where $\cK_2$ is given in Lemma~\ref{lem:gf}. Taking $\delta$ into zero, \eqref{eqn:lim21} and Theorem~\ref{thm:ap} imply that for $s,k \in \Rpz$
\begin{align}
\int_0^T \td^2(U_{t+s}^k, U_{t}^k) dt  \leq s \cK_2  (\E^\infty(k) - \E^\infty(k+s))
\end{align}
where $U^k_t$ is given in \eqref{eqn:uet}.

\medskip

Note that as $\E^\delta(k)$ is monotone decreasing for all $\delta>0$, $\E^\infty(k)$ is also monotone decreasing in $k$. Taking $k$ into $\infty$, we get for $s \in \Rpz$
\begin{align}
\int_0^T \td^2(U_{t+s}^\infty, U_{t}^\infty) dt  \leq  s \cK_2  (\inf_{k >0 } \E^\infty(k) - \inf_{k >s } \E^\infty(k)) =  0
\end{align}
and we conclude. 
 \end{proof}

\noindent {\bf Proof of Theorem~\ref{thm:cov}}

\medskip

{ 1.} Let $\ei$ and $U^{\infty}_t$ be as given in Proposition~\ref{prop:usu}. We denote $\uit$ by $U^{\infty}$ since we know that it is stationary from the last proposition. We will show that $\ei$ is independent of time as well.  Let us argue by contradiction, and suppose $\ei(t_1) \neq \ei(t_2)$  for two Lebesgue points $t_1 < t_2$ in $\Rpz$. We may assume that $\ei(t_1) < \ei(t_2)$.
As $t_1$ and $t_2$ are Lebesgue points of $\ei$, there exists $\delta_1>0$ such that for any $\delta \in (0,\delta_1)$, we have
$$
\frac{\Up(t_1 + \delta) - \Up(t_1)}{\delta} \leq \frac{ \ei(t_1) + \ei(t_2)}{2} \leq \frac{\Up(t_2 + \delta) - \Up(t_2)}{\delta} \hbox{ where } \Up(t):= \int_0^t \ei(s) ds. 
$$
Therefore, for $\delta \in (0,\delta_1)$, we have
\begin{align}
\label{eqn:cov11}
\Up(t_1 + \delta) < \Theta_1(t_1 + \delta) \hbox{ and } \Up(t_2 + \delta) > \Theta_2(t_2 + \delta)
\end{align}
where 
$$
\Theta_i(t) := \frac{1}{2} ( \ei(t_1) + \ei(t_2))(t-t_i) + \Up(t_i) \hbox{ for } i\in \{1,2\}.
$$

\medskip

From Proposition~\ref{prop:usu}, $u(x) := \chi_{U^{\infty}}(x) - \chi_{(U^{\infty})^C}(x)$
is a viscosity solution of $V= -H+\ei(t)$. Let us define $v_i: \R^n\times [0,\delta_1]\to \R$ by
\begin{align}
\label{eqn:cov13}
v_1(x,t):=\tu(x; (- \Up + \Theta_1)(t+t_1))  \hbox{ and } v_2(x,t):= \hu(x; (\Up - \Theta_2)(t+t_2))
\end{align}
Observe that $v_1=v_2=u$ at $t=0$ by definition of $\Up$ and $\Theta_i$.  Moreover by \eqref{eqn:cov11} $v_1$ and $v_2$ are each a viscosity subsolution and supersolution of $V = -H + \frac{1}{2}(\ei(t_1) + \ei(t_2))$. Hence Theorem~\ref{thm:comp} implies that $v_1 \geq v_2$ in $[0,\delta_1]$. \eqref{eqn:cov11} and $v_1(\cdot,\delta) \geq v_2(\cdot,\delta)$ can only both hold if $U^{\infty}$ is the whole $\R^n$, which is not the case here, so we reach a contradiction.
\medskip

{ 2.}
As $\uitt$ and $\ei$ are stationary from Proposition~\ref{prop:lon} and Step 1, we conclude that $\uit$ is a viscosity solution of the elliptic problem,
\begin{align}
\label{eqn:cov21}
H = \ei 
\end{align}
As $\uitt \subset S_{r_1, R_1}$ from Proposition~\ref{prop:usu}, $\uitt$ can be locally represented by graphs. Then, the regularity of \eqref{eqn:cov21} in \cite[Corollary 10.7]{GT15} implies that $U^\infty_{t_1}$ is smooth. As $\uitt \subset S_{r_1,R_1}$, we conclude that $U^\infty_{t} = B_{r_\infty}(x)$ in $\Rpz$ for some $x \in \oB_{r_1}(0)$ where $r_\infty$ given in Theorem~\ref{thm:cov}. Therefore, every sequence of $\oitt$ has a subsequence converging to $B_{r_\infty}(x)$ for some $x \in \oB_{r_1}(0)$, we conclude \eqref{prop:1lim}.
\hfill$\Box$

\medskip

\appendix

\section{Examples of Unbounded Total Mean Curvature}
\label{ap:ex}

In this section, we present examples of unbounded total mean curvature for $\R^2$ in Example~\ref{ex:tm1} and $\R^n, n \geq 3$ in Example~\ref{ex:tm2}.
Here is an example of a nonconvex domain in $\R^2$ such that the total mean curvature is unbounded although perimeter and volume are bounded.

\begin{ex}
\label{ex:tm1}
Let $\{O_{i}\}_{i \in \natural}$ be a sequence of mutually disjoint balls in $\R^2$ with radius $r_{i}:=\frac{1}{i^2}$ such that $O_{i} \subset \subset B_{10}$ \fai.
Define a domain $\Om$ by
\begin{align}
\Omega:= B_{10} - \bigcup_{i=1}^{\infty} O_i
\end{align}
Gauss-Bonnet Theorem implies $\int_{\partial O_i} H = -2\pi$ \fai. Thus, total mean curvature is unbounded although perimeter and volume are positive and bounded as follows,
\begin{align}
|\Omega| = w_n \left( 10^2 - \sum_{i=1}^{\infty} \frac{1}{i^4}  \right), \quad \per(\Omega) = n w_n \left( 10- \sum_{i=1}^{\infty} \frac{1}{i^2}  \right)
\end{align}
\end{ex}

Based on Example~\ref{ex:tm1}, we can construct a simply connected and star-shaped set, whose total mean curvature is unbounded for $n\geq 3$ (See example \ref{ex:tm2} and Lemma~\ref{lem:tm}).

\begin{ex}
\label{ex:tm2}
Let $\phi:\R^n \rightarrow \R^+$ be a smooth function such that $\phi$ is concave, 
\begin{align}
\label{ex:tm20}
0< \phi <1 \hbox{ in }B_1^{n-1}, \quad \phi = 0 \hbox{ in }(B_1^{n-1})^c \hbox{ and } |D\phi | \leq 1 \hbox{ in } \R^n. 
\end{align}
Denote
\begin{align}
\cD:= \{ [0,\phi(x')] \times x' : x' \in B_1^{n-1} \} \subset \R^n, \quad  \partial^- D:= B_1^{n-1} \text{ and } \partial^+ D:= \{ (\phi(x'),x') : x' \in B_1^{n-1} \}
\end{align}
Here, $B_1^{n-1}$ is a ball of radius $1$ and center $0$ in $\R^{n-1}$.

\medskip

As Example~\ref{ex:tm1}, choose $\{ x_i \}_{ i \in \natural } \subset \R^{n-1}$ such that $B_{r_i}^{n-1}(x_i)$ are a sequence of disjoint balls
such that 
\begin{align}
\label{ex:tm21}
r_{i}:=\frac{1}{i^{\frac{1}{n-2}}} \hbox{ and } x_i \in B_{10}^{n-1}
\end{align}
Define a domain $\Om \subset \R^n$ by
\begin{align}
\label{ex:tm22}
\Omega:= \Om_1 \cup \Om_2
\hbox{ where }
\Om_1 :=  [-40,40]^n \quad \hbox{ and } \quad \Om_2 :=   \bigcup_{i=1}^{\infty} (r_i \cD + (40, x_i)) .
\end{align}
where $e_1$ is a unit vector of the first axis. From Lemma~\ref{lem:tm}, $\Om$ is star-shaped with respect to some ball.

\medskip

The boundary is divided by two parts,
\begin{align}
\label{ex:tm23}
\pot = \Gamma_1 \cup \Gamma_2 \hbox{ where } \Gamma_1:= \left(\partial \Om_1 \setminus \left( \bigcup_{i=1}^{\infty} r_i \partial \cD^-+ (40,x_i)  \right) \right)  \hbox{ and } \Gamma_2 :=  \bigcup_{i=1}^{\infty} (r_i\partial  \cD^+ + (40,x_i)).
\end{align}
Note that $\Gamma_1$ has bounded total mean curvature, but total mean curvature of $\Gamma_2$  is unbounded as follows.
From \eqref{ex:tm21} and the change of variables, it holds that
\begin{align}
 \int_{\Gamma_2} H d\sigma = \sum_{i=1}^{\infty}  \int_{r_i\partial \cD^+ + x_i} H d\sigma,
&= \left( \int_{\partial \cD^+} H d\sigma \right) \sum_{i=1}^{\infty} r_{i}^{n-2}.
\end{align}
Then, \eqref{ex:tm23} implies that 
\begin{align}
 \int_{\Gamma_2} H d\sigma = \left( \int_{\partial \cD^+} H d\sigma \right) \sum_{i=1}^{\infty} \frac{1}{i}.
\end{align}
while volume and perimeter are bounded as follows.
\begin{align}
|\Omega| =  \left( 80^n +  |\cD| \sum_{i=1}^{\infty} r_i^{n}  \right), \quad \per(\Omega) = \left( 80^{n-1} n - \sum_{i=1}^{\infty} w_{n-1} r_i^{n-1} + |\partial^+ \cD| \sum_{i=1}^{\infty}  r_i^{n-1}  \right).
\end{align}
\end{ex}

\reservespace{
\begin{lem}
\label{lem:tm}
There exists $r>0$ such that $\Om$ given in Example~\ref{ex:tm2} is star-shaped with respect to $B_r$.
\end{lem}
}

\begin{proof}
As $\Om_1$ given in \eqref{ex:tm22} is star-shaped, it is enough to consider a point in $\Om_2$. Note that $\Gamma_2$ given in \eqref{ex:tm23} is smooth. From Lemma~\ref{lem:starloc}, it is enough to show that
\begin{align}
x \cdot \n_x \geq r
\end{align}
for some $r>0$ and all $x \in \Gamma_2$.
Denote $x = (x^1,x') \in \R^n$. For $x \in r_i \partial  \cD^+ + (40,x_i)$, there exists $y_i \in B_1^{n-1}$ such that
\begin{align}
\label{eqn:tm1}
x = \left(r_i \phi(y_i), r_i y_i\right) + (40,x_i)
\end{align}
and thus
\begin{align}
\label{eqn:tm2}
\n_x = \frac{ (1, - D\phi(y_i))}{\sqrt{ 1 + |D\phi(y_i)|^2}}.
\end{align}
From \eqref{ex:tm20} and \eqref{ex:tm21}, it holds that
\begin{align}
x \cdot \geq \frac{40}{\sqrt{ 1 + |D\phi(y_i)|^2}} - \frac{|D\phi(y_i) | (| r_i y_i| + |x_i|) }{\sqrt{ 1 + |D\phi(y_i)|^2}} \geq 10
\end{align}
\end{proof}

\medskip

\section{Proof of Proposition~\ref{prop:nom}}
\label{ap:ap}

In this section, we prove Proposition~\ref{prop:nom}.
First, in Lemma~\ref{lem:ss}, we show short-time star-shapedness based on the H\"{o}lder continuity in Lemma~\ref{lem:hol}. The remaining arguments are parallel to Theorem 6.5 in \cite{KimKwo18}. For simplicity, we fix $r_0$ and $R_0$ given in \eqref{eqn:r0} and $\delta \in (0,\delta_0)$ for $\delta_0$ given in \eqref{eqn:d0}. Also, let $\ot$ be a viscosity solution of $V = -H + \gd(|\Xi_t|)$ where $\Xi_t$ is an energy solution given in Definition~\ref{def:ene}.

\medskip

First, let us recall several properties of solutions from \cite{KimKwo18}.

\begin{lem}
\label{lem:hol}
\cite[Corollary 2.10]{KimKwo18}
	Assume that $\Omega_{0} \in S_{r,R}$. Then, there exists $$\M_1=\M_1(r,R,\|\gd(|\Xi_t|)\|_{L^\infty ([0,T])})$$ such that we have
\begin{align}
\sup_{x \in \pot} d(x, \poz) \leq \M_1|t|^{\frac{1}{2}} \hbox{ for } t \in [0,T]. 
\end{align}
\end{lem}

\begin{thm}\cite[Theorem 3.6]{KimKwo18}
\label{lem:brho}
Let $I = [0,t_0)$ be the maximal interval satisfying  $\overline{B_\rho} \subset \Omega_t$. Then, $\Omega_t$ satisfies $\rho$-\textit{reflection} in $I$.
\end{thm}

Lemma~\ref{lem:hol} and Theorem~\ref{lem:brho} imply the following lemma.

\begin{lem}
\label{lem:ss}
(Short-time star-shapedness) 
For $r>r_0>0$ and $0<R<R_0$, suppose that $\overline{B}_{(1+\beta)\rho} \subset \Omega_0$ and $\oz \in S_{r,R}$ for $r = \rho(\beta^2 + 2\beta)$. Then, for all $t \in [0,t_1]$, it holds that for some $\hat{r} > r_0$ and $\hat{R} < R_0$
\begin{align}
\label{eqn:ss2}
\ot \in S_{\hat{r},\hat{R}}.
\end{align}
where
\begin{align}
\label{eqn:t1}
t_1=t_1(r,R,\|\gd(|\Xi_t|)\|_{L^\infty ([0,T])}) := \frac{1}{2} \left(\min \left\{ \frac{ \sqrt{r^2 + \rho^2} - \sqrt{r_0^2 + \rho^2}}{\M_1}, \frac{R_0 - R}{\M_1} \right\}\right)^2
\end{align}
Here, $\M_1$ is given in Lemma~\ref{lem:hol}.
\end{lem}

\begin{proof}
From Lemma~\ref{lem:hol}, it holds that in $[0,t_1]$
\begin{align}
\label{eqn:ss1}
\overline{B}_{(1+\beta)\rho - \M_1 t_1^{\frac{1}{2}}} \subset \Omega_t \subset B_{R + \M_1 t_1^{\frac{1}{2}}} 
\end{align}
Theorem~\ref{lem:brho} implies that $\Omega_t$ satisfies $\rho$-\textit{reflection} for all $t \in [0,t_1]$

\medskip

By \eqref{eqn:t1}, it holds that
\begin{align}
R + \M_1 t_1^{\frac{1}{2}} < R_0 \hbox { and } r_0 > \left( ((1+\beta)\rho - \M_1 t_1^{\frac{1}{2}})^2 - \rho^2 \right )^{\frac{1}{2}}
\end{align}
From \eqref{eqn:ra}, we conclude.
\end{proof}

Recall definition of energy solutions and comparison principle from \cite{KimKwo18}.

\begin{DEF}
\label{def:ene}
\cite[Definition 5.2]{KimKwo18}
Let $(\Xi_t)_{t \geq 0}$ be \textit{a energy solution} if
there exists a sequence $h_k \rightarrow 0$ such that
\begin{equation*}
d_H(\Xi_t, E_{t}(h_k)) \rightarrow 0
\end{equation*}
Here, $E_t=E_t(h)$ is given in \eqref{eqn:mmov}.
\end{DEF}

\begin{lem}\cite[Proposition~6.1]{KimKwo18}
\label{lem:bar}
Suppose that $\ot \in S_{r,R}$ in $[0,T]$ for some $r>r_0$ and $R<R_0$.
\begin{align}
\hbox{ If } \Omega_0  \subset \subset \Xi_0, \hbox{ then } \Omega_t \subset\subset \Xi_t \hbox{ in }  [0,T]. \hbox{ Also, if }    \Xi_0 \subset \subset \Omega_0, \hbox{ then }\Xi_t \subset \subset \Omega_t \hbox{ in }  [0,T].
\end{align}
\end{lem}

\medskip

\noindent {\bf Proof of Proposition~\ref{prop:nom}}

\medskip

Let $\ot^{\e,+}$ and $\ot^{\e,-}$ be viscosity solutions of $V = -H + \gd(|\Xi_t|)$ starting from $\oz^{\e,-} := (1-\e)\oz$ and $\oz^{\e,+}:= (1+\e)\oz$, respectively. Note that
\begin{align}
\label{eqn:nom10}
\overline{B}_{(1+\beta_1)\rho} \subset \oz \in S_{r_1,R_1}
\end{align}
where $r_1 (= \rho(\beta_1^2 + 2\beta_1)^{\frac{1}{2}})$ and $\beta_1$ are given in from \eqref{eqn:r1} and $R_1$ from Proposition~\ref{thm:geo}.
From \eqref{eqn:nom10}, there exist $\e_0>0$, $\beta \in (0, \beta_1) $ and $R \in (R_1,R_0)$ such that 
\begin{align}
\label{eqn:nom11}
\overline{B}_{(1+\beta)\rho}  \subset \oz^{\e,\pm}  \in S_{r,R} \hbox{ for all } \e \in [0,\e_0], \hbox{ and } r :=  \rho(\beta^2 + 2\beta)^{\frac{1}{2}}>r_0
\end{align}
Here, $r_0$ and $R_0$ is given in \eqref{eqn:r0}.

\medskip

Let us show that $\Xi_t = \ot$ in $[0,t_1]$ for $t_1=t_1(r,R,K_1)$ given in \eqref{eqn:t1} and $K_1 := \frac{1}{\delta}\max\{ 1,  |B_{R_0}|\}$. As $|\gd(|\Xi_t|) | \leq K_1$, we can apply Lemma~\ref{lem:ss} combining with \eqref{eqn:nom11} and conclude that for some $\hat{r} > r_0$ and $\hat{R} < R_0$
\begin{align}
\ot^{\e,\pm} \in S_{\hat{r},\hat{R}}  \hbox{ in } [0,t_1].
\end{align}
By Lemma~\ref{lem:bar}, we conclude that 
\begin{align}
\ot^{\e,-} \subset \Xi_t \subset \ot^{\e,+}.
\end{align}
Note that $\ot^{\e,+}$ and $\ot^{\e,-}$ converges to $\ot$ as $\e \to 0$ from the uniqueness in \cite[Theorem 4.3]{KimKwo18}. We conclude that $\Xi_t = \ot$ in $[0,t_1]$. As the proof of Theorem 6.5 in \cite{KimKwo18}, we can iterate this step to conclude.
\hfill$\Box$

\medskip

\section{Geometric Properties}
\label{ap:geo}

In this section, we consider geometric properties of $\rr$ and $S_{r}$. First, let us recall a local property of $S_{r}$ from \cite{KimKwo18}.

\begin{lem}\cite[Lemma 3.2]{KimKwo18}
\label{lem:starloc}
For a continuously differentiable and bounded function $\phi : \mathbb{R}^n \rightarrow \mathbb{R}$, let us denote the positive set of $\phi$ by $\Omega(\phi)$. Let us assume that $\Omega(\phi)$ contains $B_r(0)$ and $D\phi \neq 0$ on $\partial \Omega(\phi)$. Then 
the set $\Omega(\phi)$ is in $S_r$ if and only if 
$$x \cdot \vec{n}_x = x \cdot \left(- \frac{D\phi}{|D\phi|}(x)\right) \geq r  \hbox{ for all } x\in \partial\Omega(\phi), $$
where $\vec{n}_x$ denotes the outward normal of $\partial \Omega(\phi)$ at $x$.
\end{lem}

Here are several properties of $\rr$ and $S_{r}$ from \cite{FelKim14}.

\begin{lem}
\label{lem:4rho}
\cite{FelKim14}
Suppose that $\Omega$ satisfies $\rho$\textit{-reflection}. Then, we have
\begin{equation*}
\sup_{x \in \partial \Omega} |x| - \inf_{x \in \partial \Omega} |x| \leq 4\rho.
\end{equation*}
\end{lem}

\begin{lem}\cite[Lemma 3, 9, 10]{FelKim14} \label{lem:star}
For a bounded domain $\Omega$ containing $B_r(0)$, the following are equivalent:

(i) $\Omega \in S_r$.

(ii)
For all $x\in \partial \Omega$, there is an interior cone to $\Omega$:
\begin{align}
\label{eqn:ic}
IC(x,r):=\left((x+C(-x,\theta_x)) \cap C(x,\frac{\pi}{2} - \theta_x ) \right) \cup B_r(0) \subset \Omega 
\end{align}
where $$\theta_x := \arcsin \frac{r}{|x|} \in \left[0,\frac{\pi}{2}\right] \text{ and }C(x,\theta):=\{ y \mid \langle x,y \rangle \geq \cos \theta |x| |y| \}.$$

(iii)
There exists $\epsilon>0$ such that for all $x\in \partial \Omega$, there is an exterior cone to $\Omega$:
\begin{align}
\label{eqn:ec}
EC(x,r):=\left(x+C(x,\theta_x)\right) \cap B_\epsilon(x) \subset \Omega^c \text{ where } \theta_x = \arcsin \frac{r}{|x|}. 
\end{align}
\end{lem}

Lemma~\ref{lem:den} can be shown by Lemma~\ref{lem:star}

\medskip

\noindent {\bf Proof of Lemma~\ref{lem:den}}

\medskip

From Lemma~\ref{lem:star}, it holds that for all $x \in \partial E$, 
\begin{align}
IC(x,r) \subset E \hbox{ and } EC(x,r) \subset E^c
\end{align}
where $IC$ is an interior cone given in \eqref{eqn:ic}, and $EC$ is an exterior cone given in \eqref{eqn:ec}. Note that as $|x| \leq R$, the angle of both the interior cone and exterior cone, $\theta_x$, is bounded from below as follows,
\begin{align}
\theta_x := \arcsin \frac{r}{|x|} \geq \arcsin \frac{r}{R}
\end{align}
Thus, for $\eta_1(r,R):= |IC(Re_1,r) \cap B_\e(Re_1)|$, it holds that for $\e \in (0,r)$
\begin{align}
\eta_1 \e^{n} \leq |IC(x,r) \cap B_\e(x)| \leq |E \cap B_\e(x)|  
\end{align}
Here, $e_1$ is a unit vector in the positive $x_1$ direction.
Similarly, it holds that
\begin{align}
|B_\e(x) \setminus E| \geq |B_\e(x) \cap EC(x,r)| \geq \eta_1 \e^{n}.
\end{align}
Here, $w_n$ is a volume of a unit ball in $\R^n$.

\medskip

As $E \in S_{r,R}$, there exists $\e_0 = \e(r,R) < r$ such that for all $\e \in (0,\e_0)$
\begin{align}
B_\e(x) \cap \partial E = (U, f(U)) 
\end{align}
up to rotation for some Lipschitz function $f = f_{x,\e} : U \subset B_\e^{n-1}(x) \rightarrow \R$.
Note that as $E \in S_{r,R}$, the Lipschitz constant of $f$ is uniformly bounded by some constant $\L = \L(r,R)$.

\medskip

From Theorem 9.1 in \cite{Mag12},
\begin{align}
\H^{n-1}(B_\e(x) \cap \partial E) = \int_{U} \sqrt{ 1+ | \grad f |^2 } dx \leq |U| \sqrt{ 1+ \L^2} \leq n w_n \e^{n-1}   \sqrt{ 1+ \L^2}
\end{align}
Thus, \eqref{eqn:2den} holds with $\eta_2(r,R):=n w_n   \sqrt{ 1+ \L^2}$. On the other hand, from the isoperimetric inequality in \cite[Proposition 12.37]{Mag12} and \eqref{eqn:1den}, we get the lower bound of \eqref{eqn:2den}.
\hfill$\Box$

\medskip

\begin{lem}
\label{lem:rho}
\cite[Lemma 10, 24]{FelKim14}
\begin{enumerate}
\item
Suppose that $\Omega$ satisfies $\rho$\textit{-reflection}.  
Moreover, $\Omega \in S_r$  with  
\begin{align}
\label{eqn:ra}
r=(\inf_{x\in \partial \Omega} |x|^2 - \rho^2 )^{1/2}. 
\end{align}

\item
Suppose that $\Omega$ is in $S_{r,R}$. If  there exists $\rho>0$ such that $\overline{B_\rho(0)} \subset \Omega$ and $\rho^2 \geq 5(R^2 - r^2)$, then $\Omega$ satisfies $\rho$-reflection.
\end{enumerate}
\end{lem}

For $E, F \subset \R^n$, define the Hausdorff distance by
\begin{align}
\label{eqn:dh}
d_H(E,F) := \max \left\{ \sup_{x \in E} d(x,F), \sup_{x \in F} d(x,E) \right\}
\end{align}

\begin{lem}
\label{lem-cpt}
\cite[Lemma 23]{FelKim14}
	Consider sets $\Omega_1$  and $\Omega_2$ in $S_{r,R}$ for $R>r>0$. Then the following holds:
$$ d_H(\partial \Omega_1, \partial \Omega_2) \lesssim_{r,R} d_H(\Omega_1, \Omega_2), \quad |\Omega_1 \Delta \Omega_2| \lesssim_{r,R} d_H(\Omega_1, \Omega_2),$$
	\end{lem}

Lastly, let us show the following property of characteristic functions.

\begin{lem}
\label{lem:sts}
Let $\{ \oktt \}_{k \in \natural}$ be a sequence of sets in $S_{r,R}$ for $0<r<R$. Suppose that $\okt$ converges locally uniformly to $\oit$. For a sequence of functions $\{ u_k \}_{k \in \natural \cup \{ + \infty \} }$ defined by
\begin{align}
u_k := \chi_{\okt} - \chi_{(\okt)^C} \hbox{ for } k \in \natural \cup \{ + \infty \}
\end{align}
it holds that
\begin{align}
\label{eqn:1sts}
u_\infty^* = \limsups_{k \rightarrow \infty}u_k \hbox{ and } (u_\infty)_* = \liminfs_{k \rightarrow \infty}u_k 
\end{align}
Here, $\limsups$ and $\liminfs$ are given in \eqref{eqn:sups}.
\end{lem}

\begin{proof} 
Let us show the first equation in \eqref{eqn:1sts} only. The second one can be shown by the parallel arguments.

\medskip

By uniform convergence in finite interval, for any $j \in \natural$, there exists $k_1>0$ such that for all $k>k_1$
\begin{align}
\label{eqn:sts11}
d_H(\okt, \oit) < \frac{1}{j}.
\end{align}
Thus, for any $x \in \oit$ and $k>k_1$, there exists $y \in \okt$ such that $|x-y| < \frac{1}{j}$. Thus, we conclude that 
\begin{align}
\limsups_{k \rightarrow \infty}u_k(x,t) = \lim\limits_{j \to \infty} \sup \left\{ u_k(y,s) : k \geq j,\quad |y-x| \leq \frac{1}{j}, \quad |s-t| \leq \frac{1}{j} \right\} = 1 
\end{align}
and $u_\infty^*(x) = \limsups_{k \rightarrow \infty}u_k(x)$ for $x \in \oit$.

\medskip

Note that we have for any sets $\Omega_1, \Omega_2 \in S_{r,R}$
\begin{align}
d_H(\Omega_1^C, \Omega_2^C) \leq d_H(\partial \Omega_1, \partial \Omega_2) 
\end{align}
Combining this with Lemma~\ref{lem-cpt}, we conclude that $(\okt)^C$ converges locally uniformly to $(\oit)^C$. By parallel arguments, for any $x \in (\oit)^C$, we conclude that $\limsups_{k \rightarrow \infty}u_k(x,t) = -1$. As $\limsups_{k \rightarrow \infty}u_k$ is upper semicontinuous, we conclude \eqref{eqn:1sts}.

\end{proof}

\bibliographystyle{alpha}
\bibliography{Paper_VPMCF}

\end{document}